\documentclass[12pt, final]{amsart}
\usepackage{amsmath,amsthm,amssymb,amsfonts}
\usepackage{xspace}
\usepackage[colorlinks=false, final]{hyperref}
\usepackage[scr]{rsfso}
\usepackage[shortlabels]{enumitem}
\usepackage{geometry}
\usepackage{soul}
\usepackage{multicol}
\usepackage{float}
\usepackage{array}
\usepackage{runic}
\usepackage{tikz-cd}
\DeclareMathOperator{\en}{End}
\DeclareMathOperator{\sen}{SEnd}
\DeclareMathOperator{\au}{Aut}
\DeclareMathOperator{\inn}{Inn}
\DeclareMathOperator{\Ker}{Ker}
\DeclareMathOperator{\im}{im}
\DeclareMathOperator{\rank}{rank}
\DeclareMathOperator{\Fix}{fix}
\DeclareMathOperator{\supp}{supp}
\newcommand{\tn}{\mathcal{T}_n}
\newcommand{\sn}{\mathcal{S}_n}
\newcommand{\an}{\mathcal{A}_n}
\newcommand{\os}{\textfut{A}}
\newcommand{\GL}{\mathscr{L}}
\newcommand{\GR}{\mathscr{R}}
\newcommand{\GD}{\mathscr{D}}
\newcommand{\GH}{\mathscr{H}}
\newcommand{\GJ}{\mathscr{J}}
\theoremstyle{plain}
\newtheorem{thm}{Theorem}[section]
\newtheorem{theorem}[thm]{Theorem}
\newtheorem{proposition}[thm]{Proposition}
\newtheorem{lemma}[thm]{Lemma}
\newtheorem{corollary}[thm]{Corollary}
\theoremstyle{definition}
\newtheorem{definition}[thm]{Definition}
\theoremstyle{remark}
\newtheorem{remark}[thm]{Remark}
\newtheorem{example}[thm]{Example}
\newtheorem{question}[thm]{Question}
\counterwithin{equation}{section}
\counterwithin{figure}{section}
\usetikzlibrary{decorations.pathreplacing,angles,quotes}
\title[]{The endomorphism tower of a finite symmetric group}

\author[V. Gould]{Victoria Gould}
\address{University of York}
\email{victoria.gould@york.ac.uk}

\author[A. Grau]{Ambroise Grau}
\address{No affiliation}
\email{ambroise.grau@alumni.york.ac.uk}
\author[M. Johnson]{Marianne Johnson}
\address{University of Manchester}
\email{Marianne.Johnson@manchester.ac.uk}

\author[J. Smith]{Jamie Smith}
\address{University of York}
\email{jamie.smith@york.ac.uk}
\thanks{Throughout this work, the second author was funded by Portuguese national funds through the FCT – Fundação para a Ciência e a Tecnologia, I.P., under the scope of the projects UID/00297/2025 (\url{https://doi.org/10.54499/UID/00297/2025}) and UID/PRR/00297/2025 (\url{https://doi.org/10.54499/UID/PRR/00297/2025}) (Center for Mathematics and Applications – NOVA Math).}
\begin{document}
\begin{abstract}
We consider the {\em endomorphism tower} of a monoid $M$,  that is, the sequence of monoids $\en_i(M)$ where $\en_0(M)=M$ and for all $i\geq 1$, $\en_i(M)$ is the monoid of all  endomorphisms of $\en_{i-1}(M)$. We show that for a finite monoid $M$ this sequence does not stabilise in a finite number of steps. Our focus is then on the case where $M=\sn$, the symmetric group on  a finite number $n$ of points. It is well known that other than in exceptional cases (which are avoided by taking $n \geq 7$), the corresponding {\em automorphism tower} of $\sn$ stabilises at the first step. In spite of the natural nature of this question, nothing was known of the endomorphism tower above the level $i=1$. We determine (for each $n \geq 7)$ the elements of $\en_2(\sn)$ and their multiplication and thus verify that the monoids $\en_i(\sn)$ for $i=0,1,2$ all have group of units isomorphic to $\sn$. We show that the same is true of $\en_3(\sn)$.
\end{abstract}
\maketitle

\section{Introduction}
\label{sec:intro}
The most natural examples of groups and monoids come from taking the automorphism group or endomorphism monoid of a given mathematical object. Perhaps the best known examples are those where the object is a finite set of cardinality $n$ or an $n$-dimensional vector space over a division ring $D$, yielding as groups of automorphisms the symmetric group $\sn$ and the general linear group $GL(n,D)$, respectively, and as monoids of endomorphisms the transformation monoid $\tn$ and the  monoid of matrices $M(n,D)$, respectively. Exploring the structure of such naturally occurring monoids lends itself to new areas of study within algebra. There is a substantial literature supporting this direction of study (see, for example \cite{Araujo:09, araujo:11, chubb:08, Fong:1980,  hou:07}).

A classic question posed in group theory concerns sequences of automorphism groups, as we now explain. Let $G$ be a group and let $\au(G)$ denote the group of automorphisms of $G$. We define a sequence $\au_i(G)$ inductively by putting $\au_0(G)=G$ and $\au_{i}(G)=\au(\au_{i-1}(G))$ for all $i\geq 1$. This can be extended to transfinite ordinals by setting $\au_\alpha(G)$, for a limit ordinal $\alpha$, to be the direct limit of the groups $\au_\beta(G)$, for all $\beta<\alpha$, using the fact that there is a natural morphism from $G$ to the subgroup $\inn(G)$ of $\au(G)$. This sequence is known as the {\em automorphism tower} of $G$.  One can then enquire as to the nature of the groups in this tower and, in particular, ask whether there exists an ordinal $\mu$ such that $\au_{\mu}(G)$ is isomorphic to $\au_{\nu}(G)$, for all $\nu>\mu$.  If so, then choosing the least such $\mu$, the sequence is said to {\em terminate, or stabilise,} in $\mu$ steps.  A classical result of Weilandt \cite{Weilant} shows that if $G$ is any finite centreless group then its automorphism tower terminates in a finite number of steps \cite{Weilant}. The corresponding result, showing that the automorphism tower of an infinite centreless group stabilises, possibly after an infinite number of steps, was proven by Thomas \cite{Thomas:1985} (see also \cite{Just:1999}) and extended to non-centreless groups by Hamkins \cite{hamkins:98}. It appears to be still open as to  whether the automorphism tower of an arbitrary finite group terminates after a finite number of steps.

This article is intended as an opener to the analogous study of the endomorphism tower of a monoid. Let $M$ be a monoid and let $\en(M)$ denote the monoid of semigroup endomorphisms of $M$. We define a sequence $\en_i(M)$ inductively by putting $\en_0(M)=M$ and $\en_{i}(M)=\en(\en_{i-1}(M))$ for all $i\geq 1$. This sequence is known as the {\em endomorphism tower} of $M$. (Note that the question of extending to an infinite sequence is not yet even  well posed, since there are not necessarily canonical  morphisms from $M$ to $\en(M)$ to allow us to define $\en_\alpha(M)$ for a limit ordinal $\alpha$.) One then asks, given a monoid $M$, does there exist some finite $i$ such that $\en_{i}(M)\cong\en_{i+1}(M)$, in which case we say the sequence {\em terminates, or stabilises,} in $i$ steps. If $M$ is finite, this question is easily resolved in the negative (see Proposition~\ref{prop:notstab}); the case for infinite $M$ is still open. We are led therefore to investigate other patterns and similar stability conditions for these towers. For example, under what conditions is there a `nice' pattern in the sequence of groups $\au(\en_i(M))$ for $i\geq 0$? We remark that if $G$ is the group of units of $M$, that is, the invertible elements of $M$, then there is a natural morphism from $G$ to the group of units $\au(M)$ of $\en(M)$, again given by conjugation.

The purpose of this article is to study the endomorphism tower of a finite symmetric group $\sn$. In order to avoid exceptional behaviour, we focus on the case where $n\geq 7$; it is well known that this constraint guarantees that $\au(\sn)$ is isomorphic to $\sn$, that is, the automorphism tower of $\sn$ terminates at $i=0$. The elements and the structure of $\en_1(\sn)$ have been determined and used by several authors, for example in the study of digraphs constructed via endomorphisms of groups (see \cite{Adjith:2025}), and for studying the translational hull of this endomorphism semigroup (see \cite{GGJK:2024}). As noted above, the group of units of $\en_1(\sn)$, that is, $\au(\sn)$, is isomorphic to $\sn$. We proceed to determine the elements and the structure of $\en_2(\sn)$; again, the group of units, that is, $\au(\en_1(\sn))$, is isomorphic to $\sn$. Using the structure of $\en_2(\sn)$ we then show that the group of units of $\en_3(\sn)$, that is $\au(\en_2(\sn))$, is once more isomorphic to $\sn$. Our aim is to illustrate the techniques required for the study, and to note how the complexities arise.

As mentioned earlier, this article is intended as an opener to the study of endomorphism towers. We are left with a plethora of interesting and non-trivial questions, some of which we articulate in the final section. The most immediate is that arising directly from our work: is $\au(\en_i(\sn))$ isomorphic to $\sn$ for all $i\geq 0$ and $n\geq 7$?

The structure of the paper is as follows. In Section \ref{sec:prelim} we introduce the basic definitions and results concerning semigroups, monoids and endomorphisms  required to proceed. In Section \ref{sec:gen} we show that for a finite non-trivial monoid $M$ (indeed, a finite semigroup), $M$ is never isomorphic to $\en(M)$. We also present a tool developed to decompose automorphisms of monoids. Section \ref{sec:strat} outlines our overall approach to the endomorphism tower problem for $\sn$. We recall the elements and structure of $\en_1(\sn)$. We present and verify some useful facts concerning centralisers of order $2$ elements of $\sn$ that are crucial for our later  results. Section \ref{sec:end2} contains the first tranche of very technical work, and it is here we give a full description of the elements of $\en_2(\sn)$. We proceed by first finding the group of automorphisms of $\en_1(\sn)$ and showing this is isomorphic to $\sn$. This leads us naturally to the information needed to find all the elements of $\en_2(\sn)$. In Section~\ref{sec:endomonoid} we present the Cayley table for $\en_2(\sn)$. Section~\ref{sec:AutEnd2} is again very technical; here we focus on $\en_3(\sn)$ and determine that the group of automorphisms of $\en_2(\sn)$ is again isomorphic to $\sn$. We conclude in Section~\ref{sec:thoughts} with some questions and remarks.

\section{Preliminaries}
\label{sec:prelim}
\subsection{Semigroups and endomorphisms}\label{sub:semigroups}
A semigroup is a non-empty set together with an associative binary operation, usually written as juxtaposition. A monoid is a semigroup with identity, usually denoted by $1$.

We will write maps on the right of their arguments and compose from left to right. Let $A$ be a set and let $\phi:A\to A$ be a map. We denote the image of $\phi$ by $\im(\phi)$ and the cardinality of $\im(\phi)$ by $\rank(\phi)$. The kernel of $\phi$, denoted by $\Ker(\phi)$, is the equivalence relation containing precisely those pairs $(a,b)\in A\times A$ such that $a\phi=b\phi$.

Let $S$ be a semigroup. We say that $\phi:S\rightarrow S$ is an {\em endomorphism} of $S$ if $x\phi\,y\phi=(xy)\phi$ for all $x,y\in S$; we say $\phi$ is an {\em automorphism} if, additionally, $\phi$ is bijective. If $S$ is a monoid, then it is not necessarily true that $1\phi=1$ for an endomorphism $\phi$; if this does hold, then we say that $\phi$ is a {\em monoid endomorphism}. We remark that in this article, following the lead of \cite{mazorchuk:02, Schein:98} and others, we consider endomorphisms of monoids without this extra condition. Of course, when $S$ is a group, any endomorphism {\em is} a monoid endomorphism and an automorphism of a monoid is always a monoid endomorphism. The set of all endomorphisms of $\en(S)$ is a monoid under the operation of composition of maps, and the automorphisms of $S$ form a subgroup of $\en(S)$, the group of units, denoted by $\au(S)$. We call any endomorphism that is not an automorphism a {\em singular} endomorphism and call the set of all singular endomorphisms $\sen(S)$.  Thus $\en(S)$ is the disjoint union of $\au(S)$ and $\sen(S)$. In the case where $S$ is finite (or, indeed, $\en(S)$ is finite), $\sen(S)$ is an   ideal of $\en(S)$.

\subsection{Endomorphisms preserving relations and idempotents}
\label{subsec:preserve}
One of the defining aspects of a semigroup $S$ is the existence of proper (left/right/two-sided) ideals. We make use of these ideals through Green's relations, defined as follows.
\begin{enumerate}[(i)]
	\item $x\GR y\Leftrightarrow xS^1=yS^1\Leftrightarrow(\exists a,b\in S^1)\;xa=y\text{ and }yb=x$,
	\item $x\GL y\Leftrightarrow S^1x=S^1y\Leftrightarrow(\exists a,b\in S^1)\;ax=y\text{ and }yb=x$,
	\item $x\GJ y\Leftrightarrow S^1xS^1=S^1yS^1\Leftrightarrow (\exists a,b,c,d\in S^1)\;axb=y\text{ and }cyd=x$,
	\item $x\GD y\Leftrightarrow (\exists z\in S)\;x\GR z\GL y$,
	\item $\GH=\GL\cap\GR$.
\end{enumerate}

Another defining feature of a semigroup is the structure of its idempotents (the elements $e\in S$ such that $e^2=e$). We denote the set of idempotents of a semigroup $S$ by $E(S)$ and recall that there is a \emph{natural order} on $E(S)$ defined by \[e\leq f\Leftrightarrow ef=fe=e\text{ for all }e,f \in E(S).\] For any subsemigroup $U$ of $S$ (e.g. the image of an endomorphism  of $S$), the natural order on $E(U)$ is the restriction of the natural order of idempotents of $S$ to $U$.
\begin{definition}
    Let $S$ be a semigroup and $\phi$ be an endomorphism of $S$. If $\mathscr{K}$ is a relation on $S$ then we say $\phi$ {\em preserves} $\mathscr{K}$ if for all $x,y\in S$,
        \[x\,\mathscr{K}y\Rightarrow x\phi\,\mathscr{K}y\phi\]
    and that $\phi$ {\em strongly preserves} $\mathscr{K}$ if
        \[x\,\mathscr{K}y\Leftrightarrow x\phi\,\mathscr{K}y\phi\]
Further, given a partition of a semigroup $S$ into a disjoint union of sets $\bigcup_{i\in I} X_i$, we say that an endomorphism \emph{preserves the partition} if $X_i\phi\subseteq X_i$ for all $i\in I$.
\end{definition}
The following three results follow in a straightforward manner from the definitions.
\begin{lemma}\label{lem:EndosPresGreens}
	Let $S$ be a semigroup and $\mathscr{K}$ be one of Green's relations. Then any endomorphism of $S$ preserves $\mathscr{K}$ and any automorphism strongly preserves $\mathscr{K}$.
\end{lemma}
\begin{lemma}\label{lem:imageidentityzero}
	If $e$ is an idempotent of a semigroup $S$ and $\phi$ is an endomorphism of $S$ then $e\phi$ is also an idempotent. Furthermore, $\phi$ preserves the natural order on idempotents in the sense that if $e \leq f$ then $e\phi \leq f\phi$. In particular, if $S$ has an identity element $1$ or a zero element $0$ then $1\,\phi$ is the identity of $\im(\phi)$ and $0\,\phi$ is the zero of $\im(\phi)$.
\end{lemma}
\begin{lemma}\label{lem:preimage_of_idpt_is_subsmgp}
	Let $S$ be a semigroup and let $e$ be in $E(S)$. If $\phi$ is an endomorphism of $S$ then any kernel class of $\phi$ with image $e$ is a subsemigroup of $S$.
\end{lemma}
\subsection{Automorphisms and characteristic subsets}\label{subsec:InvertibleMaps}
Let $\psi$ be an automorphism of a monoid $S$; in particular $\im(\psi)=S$. We call the $\GH$-class of $S$ containing the identity element $1$ of $S$ the \emph{group of units}. By Lemma~\ref{lem:imageidentityzero} $1\psi$ must be the identity of $\im(\psi)=S$. That is, the identity of $S$ is fixed by all automorphisms. It then follows from Lemma~\ref{lem:EndosPresGreens} that if $G$ is the group of units of $S$ then $G\psi=G$ for all automorphisms $\psi$ of $S$. Similarly, if $S$ contains a zero element $0$ then \ref{lem:imageidentityzero} gives that $\{0\}\psi = \{0\}$. The previous two situations are examples of \emph{characteristic subsets}.
\begin{definition}
	Let $S$ be a semigroup and $T$ be a subset of $S$. If for all automorphisms $\psi$ of $S$ we have $T\psi=T$ then we call $T$ a {\em characteristic subset} of $S$. If $T$ is a subsemigroup (monoid/group) then we call $T$ a characteristic subsemigroup (monoid/group).
\end{definition}
In later sections we make heavy use of the notion of a characteristic subset. Here we give some further examples.
\begin{example}\label{ex:ch}
    Let $S$ be a semigroup. The following subsets are characteristic: $E(S)$, a minimal two-sided ideal, the set of elements of $E(S)$ that are maximal (or minimal) under the partial order $\leq$. Further, if $U$ and $V$ are characteristic subsets, then so is $U\cap V$ and $U\setminus V$.
\end{example}
Now, let $M$ be a monoid and $G$ be the group of units of $M$. We can use the elements of $G$ to construct automorphisms as follows.
\begin{definition}
	If $M$ is a monoid with group of units $G$ then the automorphism of $M$ defined by $\psi_g:x\to g^{-1}xg$ is called an inner automorphism of $M$. These inner automorphisms form a subgroup of $\au(M)$, denoted $\inn(M)$. By abuse of notation, we can also use $\psi_g$ for the map of $G$ defined in the same way.
\end{definition}
Inner automorphisms are exactly the conjugation maps by elements of the group of units. To simplify our notation, for all $x\in M$ and all $g$ in the group of units $G$ we will use $x^g$ to mean $g^{-1}xg$. This is a direct generalisation of the definition of an inner automorphism of a group. Since these are defined by conjugation, the notion of a centraliser and conjugacy classes for monoids are also useful to generalise.
\begin{definition}
If $M$ is a monoid with group of units $G$ then for any element $x$ of $M$ we can define its centraliser by $G$ as follows
		\[C(x)=\{g\in G\mid x^g=x\}.\]
We do not make $G$ explicit in the notation $C(x)$ as this will always be clear from context. The conjugacy class of $x$ is then defined by
		\[x^G=\{x^g\mid g\in G\}.\]
\end{definition}
\begin{definition}
A group $G$ is said to be \emph{complete} if every automorphism of $G$ is inner and the centre of $G$ is trivial.
\end{definition}
Our interest in complete groups stems from the fact that if $G$ is a complete group, then as $G$ is centreless, $G\cong\inn(G)$ and hence $G\cong\au(G)$.
\subsection{Split short exact sequences of groups}
\label{subsec:split}
Suppose that $\mathfrak{i}:H\to G$ and $\mathfrak{r}:G\to K$ are homomorphisms of groups such that $\mathfrak{i}$ is injective and $\mathfrak{r}$ is surjective. The sequence of homomorphisms
	\[\{1\}\to H\xrightarrow{\mathfrak{i}}G\xrightarrow{\mathfrak{r}}K\to\{1\}\]
is said to be a {\em short exact sequence} of groups if $\im(\mathfrak{i})$ equals the (group-theoretic) kernel of $\mathfrak{r}$. Further if there exists a homomorphism $\mathfrak{e}:K\to G$ such that $\mathfrak{e}$ composed with $\mathfrak{r}$ is the identity on $K$, then we say that this short exact sequence \emph{splits}. In the latter case, note that for all $k \in K$ and all $h \in H$ we have
	\[(k^{-1}\mathfrak{e}\;h\mathfrak{i}\;k\mathfrak{e})\mathfrak{r}=k^{-1}\mathfrak{er}\;h\mathfrak{ir}\;k\mathfrak{er}=k^{-1}1k=1,\]
where we use the facts that $\mathfrak{er}$ is the identity and $\mathfrak{ir}$ maps all elements to $1$. Thus $k^{-1}\mathfrak{e}\;h\mathfrak{i}\;k\mathfrak{e}\in \Ker(\mathfrak{r})=\im(\mathfrak{i})$. Since $i$ is injective, for each pair $h\in H$, $k\in K$ there exists a unique element of $H$ whose image under $\mathfrak{i}$ is $k^{-1}\mathfrak{e}h\mathfrak{i}k\mathfrak{e}$; call this element $(h)\mathfrak{l}_k$. Then one can show that $\mathfrak{l}_k$ is an automorphism of $H$, giving rise to an action of $K$ on $H$ by automorphisms via the homomorphism $\mathfrak{l}:K\rightarrow\au(H)$ defined by $k\mathfrak{l}=\mathfrak{l}_k$. Thus $G$ is isomorphic to the semidirect product $G\cong K\ltimes_\mathfrak{l}H$, that is, the group whose elements are pairs $(k,h)$ with $k\in K$ and $h\in H$ with composition given by
    \[(k_1,h_1)(k_2,h_2)=(k_1k_2,[h_1(k_2\mathfrak{l})]h_2).\]
for all $h_1,h_2\in H$ and all $k_1,k_2\in K$.
\section{Endomorphisms of semigroups}
\label{sec:gen}
Before proceeding to our study of the endomorphism tower problem for symmetric groups, we present some general results and techniques. In Subsection \ref{subsec:finite} we show that endomorphism towers of finite non-trivial monoids do not stabilise. In Subsection \ref{subsec:semidirect} we show that if $M$ is a monoid whose group of units $G$ is complete then the automorphism group of $M$ is isomorphic to a direct product of the form $G\times K$ where $K$ is the subgroup of the automorphism group consisting precisely of all automorphisms that fix $G$ element-wise.
\subsection{Endomorphism towers of finite semigroups}
\label{subsec:finite}
For any non-trivial semigroup $S$ that is not idempotent free there is always at least one kind of singular endomorphism. These are the constant maps to some idempotent in $S$. That is, if $e$ is in $E(S)$ then the constant map $\phi:S\to S$ defined by $x\phi=e$ for all $x\in S$ is an endomorphism, indeed an idempotent endomorphism, of $S$. Along with these constant maps, we also have the identity map. 

Recall that the endomorphism tower of $S$ \emph{stabilises} if there exists $i$ such that $\en_i(S)\cong\en_{i+1}(S)$. In the finite case we can now use a counting argument on the idempotents to see which semigroups have an endomorphism tower that stabilises. We make use of the fact that every finite semigroup contains an idempotent.
\begin{lemma}\label{lem:EndTowerSolFin}
	If $S$ is a non-trivial finite semigroup then the size of $E(S)$ is strictly less than the size of $E(\en(S))$. Thus, $S\not\cong\en(S)$.
\end{lemma}
\begin{proof}
	Let $e$ be in $E(S)$ then the constant map to $e$ is an idempotent in $\en(S)$. So there are at least $|E(S)|$ idempotents $\en(S)$. In addition to these, there is the identity map. Thus, $|E(S)|<|E(\en(S))|$ and $S\not\cong\en(S)$.
\end{proof}
\begin{proposition}\label{prop:notstab}
	If $S$ is a finite semigroup such that the endomorphism tower of $S$ stabilises then $S$ is a trivial group and the endomorphism tower terminates at the first step.
\end{proposition}
\begin{proof}
	If $S$ is a non-trivial semigroup then we have seen in Proposition~\ref{lem:EndTowerSolFin} that $|E(\en(S))|>|E(S)|$ and $S$ is not isomorphic to $\en(S)$. Further, $\en(S)$ is also a non-trivial semigroup. It follows that the endomorphism tower of $S$ does not stabilise.
	
If $S$ is a trivial semigroup, hence a trivial group, it is clear that $\en(S)$ is trivial so that $S\cong\en(S)$ and the endomorphism tower terminates at the first step.
\end{proof}
\begin{remark}
We cannot use the above counting argument for infinite semigroups, although it is easy to find examples where cardinality problems arise. For instance, let $\mathbb{N}$ denote the non-negative integers and set $S=(\mathbb{N}, {\rm max})$. For each set $X\subseteq\mathbb{N}$ with the property that $0\in X$, define a function $\phi_X$ from $S$ to $S$ such that $a\phi_X={\rm max}\{x\mid x \leq a, x \in X\}$. It is readily verified that this is a morphism. But then $|\en(S)|\geq |\mathcal{P}(\mathbb{N})|>|\mathbb{N}|=|S|$. We leave open the question of whether there exists an infinite semigroup with stabilising endomorphism tower.
\end{remark}
\subsection{Decomposition of the automorphism group via characteristic subsets}
\label{subsec:semidirect}
The main aim of this section is to show that when the group of units of a monoid $M$ has a particular structure, we can decompose $\au(M)$ into a direct product using a particular short exact sequence.

Suppose first that $S$ is a semigroup with a characteristic subsemigroup $U$ such that each automorphism of $U$ extends to an automorphism of $S$, and that $\au(U)$ embeds into $\au(S)$ via a homomorphism
	\[\mathfrak{e}:\au(U)\to\au(S).\]
Then the restriction map $\mathfrak{r}:\au(S)\to\au(U)$, defined by $\psi\mathfrak{r}=\psi|_U$ is such that $\mathfrak{er}$ is the identity on $\au(U)$. Using this we establish the following sequence of homomorphisms, where $\mathfrak{i}$ is the inclusion map.
	\[\Ker(\mathfrak{r})\xrightarrow{\mathfrak{i}}\au(S)\xrightarrow{\mathfrak{r}}\au(U).\]
Clearly $\mathfrak{i}$ is injective, $\mathfrak{r}$ is surjective, and $\im(\mathfrak{i})=\Ker(\mathfrak{r})$. Thus, we have a split short exact sequence of groups, from which it follows that
	\[\au(S)\cong\au(U)\ltimes_\mathfrak{l}\Ker(\mathfrak{r}),\]
for some homomorphism $\mathfrak{l}:\au(U)\to\au(\Ker(\mathfrak{r}))$.

This technique is particularly powerful in the situation where $S$ is a monoid, the characteristic subgroup is taken to be the group of units, and this group is complete.
\begin{theorem}\label{Tm:AutCompleteMonoids}
If $M$ is a monoid with group of units $G$ such that $G$ is a complete group then the automorphism group of $M$ is isomorphic to $G\times K$, where $K$ is the subgroup of $\au(M)$ whose elements fix each element of $G$.
\end{theorem}
\begin{proof}
Let $M$ be a monoid with complete group of units $G$. We have seen in Subsection \ref{subsec:InvertibleMaps} that $G$ is a characteristic subset of $M$. Since $G$ is complete, all automorphisms of $G$ are inner and $\au(G)\cong G$. Each automorphism of $G$ therefore extends to an (inner) automorphism of $M$. Moreover, $\au(G)$ embeds into $\au(M)$ with image $\inn(M)$ and the restriction map $\mathfrak{r}:\au(M)\to\au(G)$, defined by $\psi\mathfrak{r}=\psi|_G$ is such that $\mathfrak{er}$ is the identity on $\au(G)$. We now have
	\[\au(M)\cong\au(G)\ltimes_\mathfrak{l}\Ker(\mathfrak{r})\cong G\ltimes_\mathfrak{l}\Ker(\mathfrak{r})\]
where $\mathfrak{l}:\au(G)\to\au(\Ker(\mathfrak{r}))$ is the homomorphism described in Subsection \ref{subsec:split} mapping each $\psi_g\in\au(G)$ to an automorphism of $\Ker(\mathfrak{r})$ which we shall (for simplicity) denote by $\mathfrak{l}_g$. Considering the construction of $\mathfrak{l}_g$ we see that for all $\omega\in\Ker(\mathfrak{r})$ we have $\omega\mathfrak{l}_g=\psi_g^{-1}\omega\psi_g$, where $\psi_g$ now denotes the inner automorphism of $M$ induced by $g$. For all $x\in M$ we have
\[x\psi_g\omega=(g^{-1}xg)\omega=(g^{-1}\omega)(x\omega)(g\omega)=g^{-1}(x\omega) g=(x\omega)\psi_g.\]

We deduce that $\mathfrak{l}$ maps each $\psi_g\in\au(G)$ to the identity morphism, and our description of $\au(M)$ is a direct product, that is, $\au(M)\cong G\times\Ker(\mathfrak{r}).$
\end{proof}
For $n \geq 7$ we recall that the symmetric group is a complete group, giving:
\begin{corollary}
	If the group of units of a monoid $M$ is isomorphic to $\sn$ for some $n\geq 7$, then $\au(M)$ is isomorphic to $\sn\times K$ for some group $K$.
\end{corollary}
We finish this section by showing that automorphisms which fix all elements in the group of units are very important regarding the preservation of centralisers.
\begin{lemma}\label{lem:SameCentUnderAut}
	Let $M$ be a monoid with group of units $G$, and let $K$ be the subgroup of $\au(M)$ whose elements fix each element of $G$. Then for any $m \in M$ and $\theta\in K$, we have that
	\[C(m) = C(m\theta).\]
\end{lemma}
\begin{proof}
	Clearly for any $m\in M$, $g\in G$ and $\theta\in K$ we have that:
		\[g\in C(m)\Leftrightarrow mg=gm,\]
	which, by applying $\theta$ on both sides gives us that
		\[m\theta g = m\theta g\theta = (mg)\theta = (gm)\theta = g\theta m\theta =g(m\theta),\]
	that is, $g\in C(m\theta)$ and $C(m)\subseteq C(m\theta)$. Since $\theta$ is an automorphism, we also get the reverse inclusion, and the equality holds.
\end{proof}
\section{Approaching the tower of the symmetric group}
\label{sec:strat}
We now begin our study of the endomorphism tower of $\sn$ for a fixed natural number $n$; to avoid discussion of degenerate behaviours, throughout we assume that $n\geq 7$. Our goal is to describe the first few monoids in this tower. In order to make clear which level of the endomorphism tower we are considering we use different characters for different levels. Elements of $[1,n]$ and elements of $\sn$ are lower case Roman $\{a,b,c,\ldots\}$, elements of $\en(\sn)$ are lower case Greek $\{\phi,\psi,\omega,\ldots\}$, elements of $\en_2(\sn)$ are upper case Greek $\{\Psi,\Phi,\Delta,\ldots\}$ and elements of $\en_3(\sn)$ are Anglo-Saxon runes $\{\textfut{A},\textfut{S},\textfut{J},\ldots\}$.
\[\setlength{\extrarowheight}{2mm}
    \begin{array}{l|*{1}{l}}
	\text{Monoid}		& \text{Characters}\\
	\hline
	\sn 					& a,b,c,\ldots\\
	\en(\sn)				& \psi,\phi,\omega,\ldots\\
	\en_2(\sn)				& \Psi,\Phi,\Delta,\ldots\\
	\en_3(\sn)				& \os,\textfut{S},\textfut{J},\ldots\\
    \end{array}\]

The main purpose of this section is to provide the necessary preliminaries concerning $\sn$ and $\en(\sn):=\en_1(\sn)$ for us to proceed in later sections with a consideration of $\en_2(\sn)$ and its group of automorphisms, that is, the group of units of $\en_3(\sn)$. To this end, in Subsection \ref{subsec:sym}, we provide a number of technical lemmas concerning elements of order $2$ in the symmetric group $\sn$. These are precisely the results that we need for subsequent examination of $\en_2(\sn)$ and $\en_3(\sn)$. In Subsection~\ref{sub:endsn} we recall from \cite{GGJK:2024} the structure of $\en(\sn)$.
\subsection{Centralisers of elements of order 2 in \texorpdfstring{$\sn$}{the symmetric group}}
\label{subsec:sym}
We start by defining terminology and notation.
\begin{definition}
    Let $g,h\in\sn$ and $i\in [1,n]$.
    \begin{enumerate}[(i)]
        \item We say that $i$ is \emph{fixed by $g$} if $ig=i$, and the \emph{fixed set of $g$}, denoted by $\Fix(g)$, is the set of all elements of $[1,n]$ fixed by $g$.
        \item The \emph{support} of $g$, denoted by $\supp(g)$, is the set of all letters appearing in the decomposition of $g$ into disjoint cycles (of length at least $2$). That is, the support of $g$ consists of all elements of $[1,n]$ which are not fixed by $g$.
        \item If $g$ has order $2$, the \emph{length} of $g$ corresponds to the number of disjoint transpositions in its decomposition into disjoint cycles, and we will denote it by $l(g)$. Note that $|\supp(g)|=2l(g)$.
        \item The elements $g$ and $h$ are said to be \emph{disjoint} if $\supp(g)\cap\supp(h)=\emptyset$.
        \item We say that $g$ is a {\em factor} of $h$ if the cycles in the decomposition of $g$ into disjoint cycles all appear in the corresponding decomposition of $h$.
    \end{enumerate}
\end{definition}
\begin{lemma}\label{lem:gktrans}
	Let $g,h\in\sn$ be elements of order two such that $h$ is not a transposition and $C(g)\leq C(h)$. Then $g$ is a factor of $h$ (in particular, $l(g) \leq l(h)$). Further, if $g$ and $h$ have the same parity, then $g=h$.
\end{lemma}
\begin{proof}
	Suppose first (for contradiction) that $g$ and $h$ are disjoint. Since $h$ is not a transposition it contains a factor $(p\;q)(u\;v)$ for some $p,q,u,v\in [1,n]$ where $p,q,u,v$ are all fixed by $g$. Thus, $(p\;u)$ is a transposition in $C(g)$ but not $C(h)$, a contradiction. Hence $\supp(g)\cap\supp(h) \neq \emptyset$.

Let $i\in\supp(g)\cap\supp(h)$ so there is a factor $(i\;j)$ in $g$ and a factor $(i\;l)$ in $h$. Further, as $(i\;j)\in C(g)$ we also have $(i\;j)\in C(h)$ and so $h^{(i\;j)}=h$. If $j\neq l$, that is, $i,j,l$ are distinct, then this implies that $(i\;l)^{(i\;j)}=(j\;l)$ is a factor of $h$ giving the contradiction $j=i$. Thus we must have $j=l$ and $(i\;j)$ is a factor of $h$.

Given any other factor $(x\;y)$ of $g$ that is distinct from $(i\;j)$, we see that $(i\;x)(j\;y)\in C(g)\subseteq C(h)$ and so $h^{(i\;x)(j\;y)}=h$. This means that $(i\;j)^{(i\;x)(j\;y)}=(x\;y)$ is a factor of $h$ and we deduce that $g$ is a factor of $h$.

  Finally, let us assume that $g$ and $h$ have the same parity. If $g\neq h$, then $h$ has factors $(s\;t)$ and $(w\;z)$ that are not factors of $g$. Since $g$ is a factor of $h$, it follows that $s,t,w,z$ are all fixed by $g$. But then $(s\;w)\in C(g)\setminus C(h)$, a contradiction. Hence, we must have that $g=h$.
\end{proof}
We now consider a situation corresponding to that of Lemma~\ref{lem:gktrans} where now $h$ is a transposition and $g$ is also odd.
\begin{lemma}\label{lem:gh-subgpcent-hIsTransp-equalOrActFully}
    Let $g,h\in\sn\setminus\an$ be elements of order two such that $h=(i\;j)$ and $C(g)\leq C(h)$. Then $g$ and $h$ are either equal or disjoint.
    
    Further, if $g\neq h$ then we must have that $\Fix(gh)=\emptyset$, that is, $gh$ acts non-trivially on all of $[1,n]$, which can only happen when $n$ is divisible by $4$.
\end{lemma}
\begin{proof}
    Let us first assume that $g$ and $h$ are not disjoint so that without loss of generality $g$ must contain a factor of the form $(i\;l)$. Thus $(i\;l)\in C(g)\subseteq C(h)$, which forces $(i\;j)^{(i\;l)}=(i\;j)$ and to avoid a contradiction we obtain that $j=l$. Thus $h=(i\;j)$ is a factor of $g$. Moreover, if we let $(p\;q)$ be another factor of $g$, we get that $(i\;p)(j\;q)\in C(g)\subseteq C(h)$ so that $(i\;j)=(i\;j)^{(i\;p)(j\;q)}=(p\;q)$, a contradiction. Therefore $g=h$ in this case.

    We now consider the case when $g$ and $h$ are disjoint, and we suppose (for contradiction) that there exists $x\in[1,n]$ such that $x\in\Fix(g)\cap \Fix(h)$. Then $(i\;x)\in C(g)\subseteq C(h)$ so that $(x\;j)=(i\;j)^{(i\;x)}=(i\;j)$, which forces $x=i$, a contradiction. Therefore every element of $[1,n]$ must be moved by either $g$ or $h$ giving $\Fix(gh)=\emptyset$. Finally, notice that if $|\supp(g)|=2m$ with $m$ odd, then we have that that $n=|\supp(g)|+|\supp(h)|=2(m+1)$ so that $n$ is divisible by $4$.
\end{proof}
Combining the previous two results gives:
\begin{corollary}\label{cor:centralisers}
	Let $g,h\in\sn$ be elements of order two having the same parity. If $C(g)=C(h)$ then $g=h$.
\end{corollary}
\begin{lemma}\label{lem:conjugateAndCentraliserImpliesEqual}
    Let $g,h\in\sn$ be conjugate elements of order two. Then $C(g)=C(h)$ if and only if $g=h$. Moreover, if $g\neq h$, then there exists $w\in\sn\setminus\an$ of order two such that $w$ is in $C(g)$ but not in $C(h)$.
\end{lemma}
\begin{proof}
For the first statement: one direction is trivial, and the other follows from Corollary \ref{cor:centralisers}. Thus we assume $g\neq h$ and exhibit an appropriate element $w$.

Suppose first that $g$ and $h$ are disjoint. If $g$ and $h$ are transpositions, then without loss of generality we can assume $g=(1\;2)$ and $h=(3\;4)$. Since $n\geq 7$, we have that $(3\;5)\in C(g)\setminus C(h)$. So, let us assume that $g$ and $h$ are a product of at least two transpositions, and $g=(1\;2)\cdots(m-1\;\;m)$ for some $m$ and let $(m+1\;\;m+2)(m+3\;m+4)$ be a factor of $h$. Then $(m+2\;\;m+3)\in C(g)\setminus C(h)$. In both cases, we get a transposition $w\in\sn\setminus\an$ such that $w\in C(g)$ but $w\notin C(h)$.

    Assume now that $X:=\supp(g)\cap\supp(h)$ is not empty. If for some $i\in X$ we have $(i\;j)$ is a factor of $g$ and $(i\;l)$ is a factor of $h$, with $j\neq l$ then $(i\;j)\in C(g)\setminus C(h)$.

    It remains to consider the case where $g$ and $h$ restrict to the same (non-identity) permutation of order two, $t$ say, on $X$. Since $g$ and $h$ are distinct conjugate elements agreeing on $X$, there exist distinct elements $u,v,x,y$ outside $X$ such that $(u\;v)$ is a factor of $g$ but not $h$ and $(x\;y)$ is a factor of $h$ but not $g$. Thus if $(p\;q)$ is a factor of $t$, we see that $(p\;u)(q\;v)(x\;y)\in C(g)\setminus C(h)$, thus completing the proof.
\end{proof}
\begin{lemma}\label{lem:each_g_has_its_nice_k}
	Let $k\in\an$ be an element of order two. Then there exists $g\in\sn\setminus\an$ of order two such that whenever $C(g)\cap C(k)=C(g)\cap C(k')$ with $k'\in\an$ of order two, we must have $k'=k$. Moreover, $g$ can be taken to be any transposition that is a factor of $k$.
\end{lemma}
\begin{proof}
Without loss of generality, assume $k=(1\;2)(3\;4)\cdots(m-1\;\;m)$, where $m$ is a multiple of $4$. Our aim is to show that $g=(1\;2)$ satisfies the statement of the lemma. Assume then that $k'\in\an$ is an element of order $2$ such that $C((1\;2))\cap C(k)=C((1\;2))\cap C(k')$; our task is to show that $k=k'$.

We start by showing that $X:=\supp(k)\cap\supp(k')$ is non-empty. For contradiction, assume that $X=\emptyset$. Since $k'\in\an$ is order two it has a factor $(i\;j)(p\;q)$ which is disjoint from $k$, and hence in particular disjoint from $g$. Thus, $(i\;p)$ is in $C(g)\cap C(k)$ but not in $C(k')$, a contradiction. Thus, $X\neq\emptyset$.

We now show that $X\setminus\{1,2\}$ is non-empty. Suppose this is not the case. Then, as $X\neq \emptyset$, we have that $X$ contains at least one of $1$ or $2$. If $1\in X$, then there is a factor $(1\;i)$ of $k'$. Since we have that $(1\;2)\in C(g)\cap C(k)=C(g)\cap C(k')$, we get that $(1\;i)^{(1\;2)}$ is a factor of $k'$, which happens if and only if $i=2$. Together with the dual argument for $2\in X$, we deduce that $X=\{1,2\}$ and $(1\;2)$ is a factor of $k'$. Moreover, as $k'\in\an$ is order two it has a factor $(1\;2)(u\;v)$ for some $u,v\in\Fix(k)$. If $k$ fixes a third value, say $w$, then $(u\;v\;w)$ is in $C(g)\cap C(k)$ but not in $C(k')$, a contradiction. This yields $\Fix(k)=\{u,v\}$ and since $k=(1\;2)(3\;4)\cdots(m-1\;\;m)$ we deduce that $n$ is even, $m=n-2$ and $(u\;v)=(n-1\;\;n)$. Since $X=\{1,2\}$ it follows that $k'=(1\;2)(n-1\;\;n)$. Using the fact that $n\geq 7$ we have that $(4\;5)$ is in $C(g)\cap C(k')$ but not in $C(k)$, a contradiction. Thus, $X\setminus\{1,2\}$ must be non-empty.

Our next step is to show that $k$ and $k'$ act the same on $X\setminus\{1,2\}$. We begin by showing that $k$ and $k'$ share a common factor disjoint from $(1\;2)$. Let $x\in X\setminus\{1,2\}$ so that $k$ has a factor $(x\;y)$ and $k'$ has a factor $(x\;z)$. If $z\neq y$ then $(x\;y)$ is in $C(g)\cap C(k)$ but not in $C(k')$, a contradiction. So $k$ and $k'$ share the factor $(x\;y)$.

We now use our common factor to show that any factor of $k$ other than $(1\;2)$ is a factor of $k'$. Let $(i\;j)$ be a factor of $k$ other than $(1\;2)$ and let $(x\;y)$ be a factor shared by $k$ and $k'$. Then $(i\;x)(j\;y)$ is in $C(g)\cap C(k)=C(g)\cap C(k')$. Thus, $(x\;y)^{(i\;x)(j\;y)}=(i\;j)$ must be a factor of $k'$. Dually, any factor of $k'$ other than $(1\;2)$ must be a factor of $k$.

Now, if $(1\;2)$ is not a factor of $k'$ then $k'=(3\;4)(5\;6)\cdots(m-1\;\;m)$. However, this is a product of an odd number of transpositions, which contradicts $k'\in\an$. Thus, we must have that $(1\;2)$ is a factor of $k'$ so that $k'=k$, as required.
\end{proof}
\subsection{The endomorphism monoid of the symmetric group}\label{sub:endsn}
For convenience we recall the description of $\en(\sn)$, which may be found in \cite{GGJK:2024}. It is well known that for any $n\geq 7$ the automorphisms of $\sn$ are all inner, and since $\sn$ is centreless, $\sn$ is complete and consequently $\sn\cong\inn(\sn)\cong\au(\sn)$. Since $\sn$ is a group, for any endomorphism of $\sn$, the kernel-class of the identity must form a normal subgroup and since $n\geq 7$, the only proper non-trivial normal subgroup of $\sn$ is $\an$. From these observations it is then easy to see that the endomorphisms of $\sn$ must have one of the following forms.
\begin{enumerate}[(i)]
	\item The automorphisms, where the kernel-class of $1$ is $\{1\}$:
		 \[\psi_g:x\mapsto g^{-1}xg,\]
	where $g$ is in $\sn$.
	\item The rank 2 endomorphisms, where the kernel-class of $1$ is $\an$:
		\[\phi_t:x\mapsto\begin{cases}
		t	&\text{if }x\in\sn\setminus\an\\
		1	&\text{if }x\in\an,
	\end{cases}\]
	where $t$ is an order two element of $\sn$.
	\item The rank 1 endomorphism, where the kernel class of $1$ is $\sn$:
		\[\phi_1:x\mapsto 1.\]
\end{enumerate}
In fact, each of the singular endomorphisms is uniquely determined by its image, which is highly unusual for monoids of transformations.

These elements of $\en(\sn)$ compose as follows for all $g,h\in\sn$ and all $t,s\in\sn$ with $t^2=1=s^2$:
\[\psi_g\psi_h=\psi_{gh}, \qquad \psi_g\phi_t=\phi_t, \qquad \phi_t\psi_g=\phi_{t^g}, \qquad \phi_s\phi_t=\begin{cases}
        \phi_{t} &\text{if } s\in\sn\setminus\an\\
        \phi_{1} &\text{if } s\in\an.
    \end{cases}\]

Now, following the example of \cite{GGJK:2024}, we can partition the monoid $\en(\sn)$ by rank and type as follows: 
\begin{enumerate}[(i)]
    \item $\au(\sn)$;\hfill (group of units)
    \item $E=\{\phi_g\mid g\in\sn\setminus\an,\;g^2=1\}$;\hfill (idempotents of rank $2$)
    \item $A=\{\phi_g\mid g\in\an,\;g\neq 1,\;g^2=1\}$;\hfill (non-idempotents of rank $2$)
    \item $\{\phi_1\}$.\hfill (the zero)
\end{enumerate}
\begin{remark}\label{rem:ensnbits}
	We have
\[E(\en(\sn))=\{\psi_1\}\cup E\cup\{\phi_1\}.\] The set $\au(\sn)$ is the group of units. The set $E$ is a right-zero band, that is, a semigroup in which the product of any two elements is the second; further, the elements of $E$ are left identities for elements of $E\cup A\cup\{\phi_1\}$. Finally, $A\cup\{\phi_1\}$ forms a nil-semigroup, that is, the product of any two elements is the zero, $\phi_1$.
\end{remark}

We now observe that the conjugacy classes of singular elements of $\en(\sn)$ are determined by conjugacy in $\sn$. This pattern is repeated at higher levels.
\begin{remark}\label{rem:conjendsn}
The sets $E$ and $A$ are unions of conjugacy classes of singular elements $\phi_g$ for some $g\in\sn$ where $g$ is even of order 2. The conjugacy class of such $\phi_g\in E\cup A\cup \{\phi_1\}$ is equal to the set product
	\[\phi_g\au(\sn)=\{\phi_g\psi_h\mid h\in\sn\}=\{\phi_{g^h}\mid h\in\sn\}.\]
Consequently, $\phi_g$ and $\phi_h$ lie in the same conjugacy class of $\en(\sn)$ if and only if $g$ and $h$ are conjugate in $\sn$, that is, if and only if $l(g)=l(h)$.
\end{remark}
Similarly, the centralisers of singular elements of $\en(\sn)$ are determined by centralisers in $\sn$.
\begin{remark}\label{rem:centraliserendsn}
Let $g,k\in\sn$ with $g^2=1$ so that $\phi_g,\psi_k\in\en(\sn)$. Then
    \[\begin{array}{rcl}
    \psi_k\in C(\phi_g)&\Leftrightarrow& \psi_k^{-1}\phi_g\psi_k=\phi_g\\
    &\Leftrightarrow& \phi_{g^k}=\phi_g\psi_k=\phi_g\\
    &\Leftrightarrow& g^k=g\\
    &\Leftrightarrow& k\in C(g).\end{array}\]
\end{remark}

Putting together Remark~\ref{rem:conjendsn}, Remark~\ref{rem:centraliserendsn} and Lemma~\ref{lem:conjugateAndCentraliserImpliesEqual} now gives:
\begin{corollary}\label{lem:SameCentPart2} 
Let $\phi_g,\phi_h\in\en(\sn)$ be in the same conjugacy class. Then $C(\phi_g)=C(\phi_h)$ if and only if $\phi_g=\phi_h$ if and only if $g=h$.
\end{corollary}
\begin{proof}
By Remark \ref{rem:conjendsn}, $g$ and $h$ are conjugate. Then by Remark \ref{rem:centraliserendsn}, $C(\phi_g)=C(\phi_h)$ if and only if $C(g)=C(h)$, which by Lemma \ref{lem:conjugateAndCentraliserImpliesEqual} is the case if and only if $g=h$, or equivalently $\phi_g=\phi_h$.
\end{proof}
We finish this section with some useful remarks from \cite{GGJK:2024} on Green's relations in $\en(\sn)$. Let $\alpha,\;\beta\in\en(\sn)$. Then $\alpha\,\GL\,\beta$ if and only if either both are units or $\alpha=\beta$. On the other hand, $\alpha\,\GR\,\beta$ if and only if one of the following holds:
\begin{enumerate}[(i)]
	\item $\alpha$ and $\beta$ are both units,
	\item $\alpha,\beta\in E$,
	\item $\alpha,\beta\in A$ and $\alpha\au(\sn)=\beta\au(\sn)$,or
	\item $\alpha=\beta=\phi_1$.
\end{enumerate}
It is then clear that $\GH=\GL$ and $\GR=\GD=\GJ$, the final equality holding since $\en(\sn)$ is finite.

\begin{figure}
	\begin{tikzpicture}
		\node[draw,align=center] at (0,5) {$\au(\sn)=\{\psi_g\mid g\in\sn\}$};
		\node[draw] at (-3,2) {$E=\{\phi_g\mid g\in\sn\setminus\an\}$};
		\node[draw] at (3,2) {$A=\{\phi_g\mid g\in\an\setminus\{\phi_1\}\}$};
		\node[draw] at (0,-1) {$\{\phi_1\}$};
		\draw [thick, decorate,decoration={brace,amplitude=10pt},yshift=2pt] (-2.3,5.4) -- (2.3,5.4) node [black,midway,yshift=16pt] {Group of Units};
		\draw [thick, decorate,decoration={brace,amplitude=10pt},yshift=2pt] (-5.2,2.4) -- (-0.8,2.4) node [black,midway,yshift=16pt] {Right zero band};
		\draw [thick, decorate,decoration={brace,amplitude=10pt},yshift=2pt] (0.6,2.4) -- (5.4,2.4) node [black,midway,yshift=16pt] {Nilpotent rank 2};
		\draw [thick, decorate,decoration={brace,amplitude=10pt},yshift=2pt] (-0.55,-0.6) -- (0.55,-0.6) node [black,midway,yshift=16pt] {Zero Element};
		\draw [<->] (6,6.2) -- (6,-1.5) node[midway, right] {rank};
		\draw [->,dashed] (-2,4.5) -- (-3,3.2) node[midway, right] {{\tiny $\leq_{\GJ}$}};
		\draw [->,dashed] (-0.7,2) -- (0.5,2) node[midway, above] {{\tiny $\leq_{\GJ}$}};
		\draw [->,dashed] (0.7,1.6) -- (0,0.4) node[midway, right] {{\tiny $\leq_{\GJ}$}};
	\end{tikzpicture}
\label{Table}
\caption{The blocks of $\en(\sn)$ arranged by rank vertically and type horizontally. Dashed arrows indicate the quasi-order $\leq_{\GJ}$ associated with the relation $\GJ$. (The set $A$ shatters into several $\GJ$-classes.)}
\end{figure}
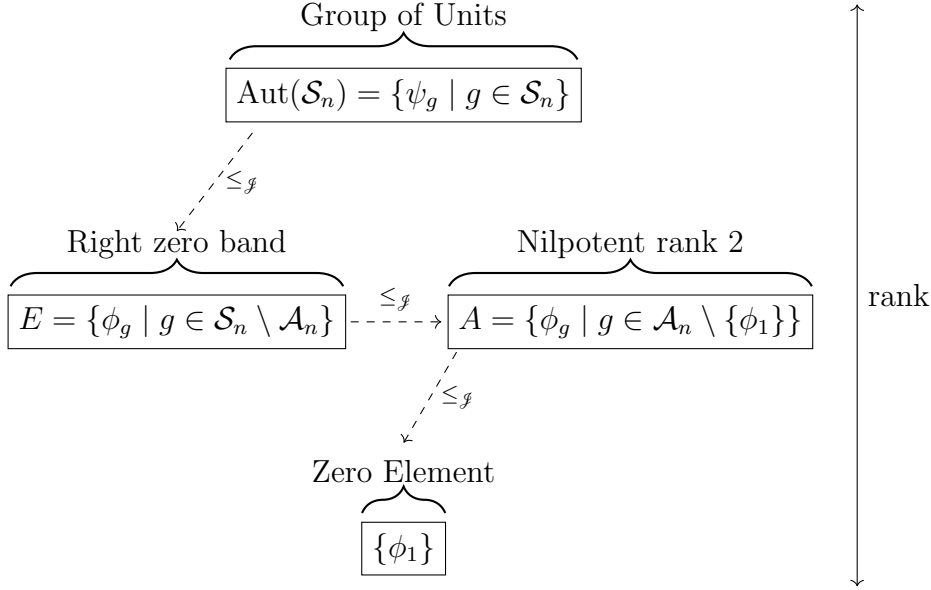
\newpage
\section{Endomorphisms of \texorpdfstring{$\en(\sn)$}{} - the elements of \texorpdfstring{$\en_2(\sn)$}{}}
\label{sec:end2}
In this section we will describe the elements of the next level, that is, we aim to find all elements of $\en_2(\sn)$. To begin with, we make some useful observations. Let $\Phi\in\en_2(\sn)$. As $\en(\sn)$ has identity $\psi_1$ and zero $\phi_1$, it follows from Lemma~\ref{lem:imageidentityzero} that $\im(\Phi)$ is a monoid subsemigroup of $\en(\sn)$ and has identity $\psi_1\Phi$ and zero $\phi_1\Phi$. From this, if $\psi_1\Phi= \phi_1\Phi$, then $\Phi$ must be a constant map. Similarly, if $\phi_1\Phi=\psi_1$, then for all $\theta\in\en(\sn)$ this forces $\theta\Phi=\theta\Phi\phi_1\Phi=(\theta\phi_1)\Phi=\psi_1$ and so $\Phi$ is also constant under this condition. Moreover, all the endomorphisms of rank $1$, that is, all the constant maps can be easily described.
\begin{lemma}\label{lem:EndosRank1}
	Given $\varepsilon$ an idempotent of $\en(\sn)$, define by $\Gamma_\varepsilon$ the map on $\en(\sn)$ sending all $\omega\in\en(\sn)$ to $\varepsilon$. Then $\Gamma_\varepsilon$ is a constant endomorphism of $\en(\sn)$.
	
	Conversely, if $\Phi$ is an endomorphism of $\en(\sn)$ with $\rank(\Phi)=1$, then $\Phi=\Gamma_\varepsilon$ for some $\varepsilon\in E(\en(\sn))$.
\end{lemma}
From now on, the symbol $\Gamma$ will be reserved for such endomorphisms of rank $1$. Before describing all the other endomorphisms of $\en(\sn)$, it is worth noting that the image of the group of units $\au(\sn)$ is quite constrained, as given by:
\begin{lemma}\label{lem:image_of_AutSn}
	Let $\Phi$ be an endomorphism of $\en(\sn)$. Then $\au(\sn)\Phi$ has size either $1$, $2$ or $n!$ and is one of the following:
	\begin{enumerate}[(i)]
		\item $\{\varepsilon\}$ for some $\varepsilon\in\{\psi_1\}\cup E\cup \{\phi_1\}$;
		\item $\{\psi_g,\psi_1\}$ for some $g\in\sn\setminus\{1\}$ of order two;
		\item $\au(\sn)$, so that $\Phi$ is a bijection on $\au(\sn)$.
	\end{enumerate}
	Moreover, in case $(ii)$, $\au(\sn)$ splits into two kernel classes: $\{\psi_h\mid h\in\sn\setminus\an\}$ and $\{\psi_h\mid h\in\an\}$.
\end{lemma}
\begin{proof}
	Notice first that $\au(\sn)$ is a $\GH$-class of $\en(\sn)$ and by Lemma~\ref{lem:EndosPresGreens} we must have that its image under $\Phi$ lies in a $\GH$-class. By the previous description of the Green's relations of $\en(\sn)$, it follows that either $\au(\sn)\Phi$ is a singleton in $E\cup A\cup\{\phi_1\}$, in which case the element in the image must be an idempotent, or we have that $\au(\sn)\Phi\subseteq\au(\sn)$. In the latter case, since $\au(\sn)$ is isomorphic to $\sn$ its only non-trivial normal subgroup is isomorphic to $\an$, and thus either $\au(\sn)$ is a single kernel class, and then $\au(\sn)\Phi=\{\psi_1\}$, or $\au(\sn)$ splits into singleton kernel classes, in which case $\Phi$ acts bijectively on $\au(\sn)$, or we have that $\au(\sn)$ splits into two kernel classes. In that last case, the image must be a subgroup of $\au(\sn)$ of order $2$, and thus must contain $\psi_1$ and a distinct element $\psi_g$ that squares to $\psi_1$, which forces $g$ to be of order two. Since we must have that $\psi_1\Phi=\psi_1$, it follows that the two kernel classes are $\{\psi_h\mid h\in\an\}$ and $\{\psi_h\mid h\in\sn\setminus\an\}$.
\end{proof}

With this in place, we will describe all the endomorphisms of $\en(\sn)$ by first considering those which are automorphisms, and then splitting the search of the singular endomorphisms according to ranks. Endomorphisms of rank $2$ will be dealt separately but those of larger ranks will be investigated relative to the size of the image of $\au(\sn)$.
\subsection{The automorphisms of \texorpdfstring{$\en(\sn)$}{}}
\label{subsec:aut}
We start by noticing that the group of units of $\en(\sn)$ forms a group isomorphic to the symmetric group, which is a complete group. Thus, by Theorem~\ref{Tm:AutCompleteMonoids} we have
\begin{equation}\label{eq:directprod}
	\au(\en(\sn))\cong\Ker(\mathfrak{r})\times\sn,
\end{equation}
where $\Ker(\mathfrak{r})$ consists of all elements $\Psi$ in $\au(\en(\sn))$ such that $\psi_g\Psi=\psi_g$ for all $\psi_g\in\au(\sn)$. That is, we have simplified the problem to finding the automorphisms of $\en(\sn)$ that fix all of the invertible elements. Since every $\Psi$ in $\Ker(\mathfrak{r})$ fixes the unit elements, we only need to describe how any $\Psi\in\Ker(\mathfrak{r})$ acts on $\sen(\sn)$. We will leverage the fact that we know how $\Psi$ acts on $\au(\sn)$ by using the centralisers of our singular elements.
\begin{lemma}\label{lem:autPreservesPartition}
	The sets $\au(\sn), E, A$ and $\{\phi_1\}$ of our partition are characteristic.
\end{lemma}
\begin{proof}
	Recall from Subsection~\ref{subsec:InvertibleMaps} that the group of units and the zero are always characteristic in finite semigroups. So we only need to show that $E$ and $A$ are characteristic. Notice that $E$ contains only idempotents and that $A$ contains no idempotents. Thus, as $\au(\sn)$ and $\phi_1$ are already characteristic, any automorphism $\Psi$ cannot map any element of $E$ into $A$. Thus $E$ is characteristic.	Now, given that $\en(\sn)\setminus A$ is characteristic, we also have that $A$ is characteristic.
\end{proof}
\begin{theorem}\label{tm:StructureOfAutEndSn}
	The automorphism group of $\en(\sn)$ is isomorphic to $\sn$. It consists only of inner automorphisms $\Psi_g:\omega\mapsto\psi_g^{-1}\omega\psi_g$. Consequently, each conjugacy class in $\en(\sn)$ is characteristic.
\end{theorem}
\begin{proof}
By Equation \ref{eq:directprod} it suffices to show that if $\Psi\in\Ker(\mathfrak{r})$, then $\Psi$ is the identity. By definition, we have that $\Psi$ fixes each automorphism. Since $\{\phi_1\}$ is characteristic, it therefore suffices to show that each automorphism in $\Ker(\mathfrak{r})$ also fixes every element of $E\cup A$.

Let $\phi_g\in E\cup A$ and $\Psi\in\Ker(\mathfrak{r})$. By Lemma~\ref{lem:autPreservesPartition} we know that $E$ and $A$ are characteristic, so that $\phi_g\Psi=\phi_k$ for some $k$ of order two of the same parity as $g$. Moreover, by applying the general theory of Lemma~\ref{lem:SameCentUnderAut} to $M=\en(\sn)$ we also have that $C(\phi_g)=C(\phi_k)$. We aim to show that $\phi_k$ is conjugate to $\phi_g$, so that Lemma~\ref{lem:SameCentPart2} will yield $\phi_g=\phi_k$.

    Suppose then (for contradiction) that $\phi_k$ is not conjugate to $\phi_g$. Then, we know that $l(g)\neq l(k)$ and since $\Psi$ is an automorphism we can assume, without loss of generality, that $l(g)<l(k)$. Thus $l(g)$ and $l(k)$ must be both odd or both even, so that $l(g)+2\leq l(k)$. However, this means that there exists three elements in $[1,n]$ fixed by $g$ but not $k$ and by letting $a\in\sn$ be a 3-cycle composed by these elements, we get that $\psi_a\in C(\phi_g)\setminus C(\phi_k)$, a contradiction.
\end{proof}
\subsection{Endomorphisms and the partition}
\label{subsec:partition}
In order to describe the non-constant endomorphisms of $\en(\sn)$, we must first understand some ideas about their kernel classes and about the role the partition $\au(\sn)\cup E\cup A\cup\{\phi_1\}$ of $\en(\sn)$ plays in restricting the situations that can be encountered. Recall that when each block of our partition is mapped to itself, we will say that the map \emph{preserves the partition}. Throughout the following lemmas, we consider how a given part of the partition not being preserved can influence the kernel classes as well as the image of $\phi_1$ which will be of particular interest.
\begin{lemma}\label{ImPhi1}
	Let $\Phi$ be an endomorphism of $\en(\sn)$ such that $\phi_1\Phi\neq\phi_1$. Then one of the following situations holds:
	\begin{enumerate}[(i)]
		\item $\phi_1\Phi=\psi_1$, and then $\Phi$ is the constant map $\Gamma_{\psi_1}$;
		\item $\phi_1\Phi=\phi_k$ for some $\phi_k\in E$ and $E\Phi=\{ \psi_1\}$, and then the kernel classes of $\Phi$ are $\au(\sn)\cup E$ and $A\cup\{\phi_1\}$;
		\item $\phi_1\Phi=\phi_k$ for some $\phi_k\in E$ and $E\Phi=\{\phi_k\}$, and then $\sen(\sn)\Phi=\{\phi_k\}$.
	\end{enumerate}
\end{lemma}
\begin{proof}
	Let $\Phi$ be as given. Since endomorphisms map idempotents to idempotents, we have two possibilities for the image of $\phi_1$. If $\phi_1\Phi=\psi_1$, then by Lemma~\ref{lem:EndosRank1} and the remarks preceding it, we must have that $\Phi=\Gamma_{\psi_1}$ so that we are in case $(i)$.
	
	Suppose then that $\phi_1\Phi=\phi_k$ for some $\phi_k\in E$. Let $\phi_g\in A$. Then we must have $\phi_k=\phi_1\Phi=(\phi_g^2)\Phi=(\phi_g\Phi)^2$. Thus $\phi_g\Phi=\phi_k=\phi_1\Phi$ and $A\cup\{\phi_1\}$ is contained in a kernel class of $\Phi$.

    Now consider $\phi_h\in E$. Since $\phi_1<\phi_h$, we have $\phi_k=\phi_1\Phi\leq\phi_h\Phi$ by Lemma~\ref{lem:imageidentityzero}. As the idempotents in $E$ are incomparable, we get that either $\phi_h\Phi=\phi_k$ or $\phi_h\Phi=\psi_1$. On the other hand, all elements of $E$ are $\GR$-related so that by Lemma~\ref{lem:EndosPresGreens} $E\Phi$ lies in a single $\GR$-class, which means that $E\Phi=\{\psi_1\}$ or $E\Phi=\{\phi_k\}$.

    If $E\Phi=\{\psi_1\}$, then for any $\phi_h\in E$ and $\psi_s\in\au(\sn)$, we have that
	\[\psi_s\Phi=(\psi_s\Phi)\psi_1=(\psi_s\Phi)(\phi_h\Phi)=(\psi_s\phi_h)\Phi=\phi_h\Phi=\psi_1,\]
so that $\au(\sn)\cup E$ is a kernel class distinct from $A\cup\{\phi_1\}$, which corresponds to case $(ii)$. Finally, if $E\Phi=\{\phi_k\}$. Then $E\Phi=A\Phi=\phi_1\Phi$ so that the singular part of $\en(\sn)$ is contained in a kernel class and case $(iii)$ holds.
\end{proof}
\begin{lemma}\label{Asubset}
    Let $\Phi$ be an endomorphism of $\en(\sn)$ such that $A\Phi\not\subseteq A\cup\{\phi_1\}$.
	Then $\phi_1\Phi\neq\phi_1$.
\end{lemma}
\begin{proof} We note that for any $\phi_k\in A$, $(\phi_k\Phi)^2=\phi_k^2\Phi=\phi_1\Phi$, so $\phi_k\Phi$ must square to the zero $\phi_1\Phi$ of $\im(\Phi)$. Thus as $A\Phi\not\subseteq A\cup\{\phi_1\}$, there exists $\phi_k \in A$ such that $\phi_k\Phi\in E\cup\au(\sn)$, and so $\phi_k\Phi$ cannot square to $\phi_1$, that is $\phi_1\Phi\neq\phi_1$.
\end{proof}
\begin{lemma}\label{lem:phi1}
	Let $\Phi$ be an endomorphism of $\en(\sn)$ such that $\phi_1\Phi=\phi_1$ and suppose that $E\Phi\not\subseteq E$. Then either
    \begin{enumerate}[(i)]
    \item $E\Phi=\sen(\sn)\Phi=\{\phi_1\}$; or
    \item $(E\cup\au(\sn)\Phi=\{\psi_1\}$
and $(A\cup\{\phi_1\})\Phi=\{\phi_1\}$.
    \end{enumerate}
\end{lemma}
\begin{proof}
	Let $\phi_g\in E$ be such that $\phi_g\Phi\not\in E$.
	If $\phi_g\Phi=\phi_1$, then for all $\phi_k\in\sen(\sn)$ we have that
    	\[\phi_k\Phi=(\phi_g\phi_k)\Phi=\phi_g\Phi\phi_k\Phi=\phi_1(\phi_k\Phi)=\phi_1,\]
	so that we get case $(i)$.

	Otherwise, since $\phi_g$ is idempotent, we must have that $\phi_g\Phi=\psi_1$, and
    for any $\alpha\in\au(\sn)\cup E$ we get that
	\[\alpha\Phi=\alpha\Phi\psi_1=\alpha\Phi\phi_g\Phi=(\alpha\phi_g)\Phi=\phi_g\Phi=\psi_1,\]
	so that $(\au(\sn)\cup E)\Phi=\{\psi_1\}$. Finally, for any $\beta\in A \cup\{\phi_1\}$ we have $\beta\phi_g=\phi_1$. Thus $(A\cup\{\phi_1\})\Phi=\{\phi_1\}$, which corresponds to case $(ii)$.
\end{proof}
If a function on the singular part preserves the `parity' of each $\phi_g$ and fixes the zero element, then it is forced to be an endomorphism of the singular part:
\begin{proposition}\label{prop:SingPrev}
    If $\Phi$ is a map on $\en(\sn)$ such that $E\Phi\subseteq E$, $A\Phi\subseteq A\cup\{\phi_1\}$ and $\phi_1\Phi=\phi_1$, then $\phi_g\Phi\phi_h\Phi=(\phi_g\phi_h)\Phi$ for all elements $\phi_g,\phi_h\in\sen(\sn)$.
\end{proposition}
\begin{proof}
	The properties of $\Phi$ mean that for all $\phi_g,\phi_h\in\sen(\sn)$, we have the following two facts:
	\begin{enumerate}[(i)]
		\item $\phi_g$ is a left identity of $\phi_h$ if and only if $\phi_g\Phi$ is a left identity for $\phi_h\Phi$; and
		\item $\phi_h$ is a right zero of $\phi_g$ if and only if $\phi_h\Phi$ is a right zero for $\phi_g\Phi$.
	\end{enumerate}
	The proposition then follows by noticing that any possible product will always involve one of the above facts.
\end{proof}
\subsection{The rank two endomorphisms}
\label{subsec:rank2}
We have already encountered in Lemmas~\ref{ImPhi1} and~\ref{lem:phi1} some endomorphisms of rank $2$.
In order to find all the possible endomorphisms of rank $2$, we are looking for congruences with two classes, and for semilattices of size $2$ in $\en(\sn)$, since by picking one of these congruences and one of these semilattices we can create a map with kernel and image the given congruence and semilattice respectively, which will give us an endomorphism of rank $2$.
\begin{lemma}\label{lem:imPhi_contains_psi1_or_phi1}
    Let $\Phi$ be an endomorphism of $\en(\sn)$ with $\rank(\Phi)\geq 2$. Then $\im(\Phi)$ must contain at least one of $\psi_1$ and $\phi_1$.
\end{lemma}
\begin{proof}
	Suppose that $\im(\Phi)$ contains neither $\psi_1$ or $\phi_1$. Since we know that $\psi_1\Phi$ and $\phi_1\Phi$ must be idempotent, it follows that $\psi_1\Phi = \phi_g$ and $\phi_1\Phi = \phi_h$ for some $\phi_g, \phi_h\in E$. Moreover, as $\rank(\Phi)\geq 2$, by the remarks preceding Lemma~\ref{lem:EndosRank1} we must also have that $\phi_g\neq\phi_h$. But then, $\psi_1\Phi\phi_1\Phi=\phi_g\phi_h=\phi_g$ while $\phi_1\Phi\psi_1\Phi=\phi_h\phi_g=\phi_h$, which contradicts the fact that $\psi_1\Phi$ is the identity of $\im(\Phi)$.
\end{proof}
\begin{theorem}\label{thm:rank2}
	Let $\Omega$ be a rank 2 endomorphism of $\en(\sn)$. Then the image of $\Omega$ is a two element semilattice, that is, the image is equal to one of: $\{\psi_1,\phi_g\}$, $\{\phi_g,\phi_1\}$ or $\{\psi_1,\phi_1\}$ for some $\phi_g\in E$. The kernel classes of $\Omega$ are either $\au(\sn)$ and $\sen(\sn)$; or $\au(\sn)\cup E$ and $A\cup\{\phi_1\}$. Moreover, any combination of the listed possibilities for the image and kernel classes corresponds to an endomorphism.
\end{theorem}
\begin{proof}
    Let $\Omega$ be a rank $2$ endomorphism of $\en(\sn)$. We know that the image of $\Omega$ must have identity $\psi_1\Omega$ and zero $\phi_1\Omega$ so that the image contains exactly two idempotents.
	Moreover, by Lemma~\ref{lem:imPhi_contains_psi1_or_phi1} at least one of $\psi_1$ or $\phi_1$ is in the image. It then follows that the possible images are $\{\psi_1,\phi_g\}$, $\{\phi_g,\phi_1\}$, or $\{\psi_1,\phi_1\}$ where $\phi_g$ is in $E$. It remains to determine the kernel classes of $\Omega$.

    By Lemma~\ref{lem:EndosPresGreens}, if two elements are $\GR$-related then their images under $\Omega$ are $\GR$-related. Thus, $\au(\sn)$ and $E$ must each be contained in a kernel classes of $\Omega$.

    Let $\phi_k\in A$. As the image of $\Omega$ contains only idempotents, $\phi_k\Omega=(\phi_k\Omega)^2=(\phi_k)^2\Omega=\phi_1\Omega$. Thus, $A\cup\{\phi_1\}$ must be contained in a kernel class of $\Omega$.

    So we have that $\au(\sn)$, $E$, and $A\cup\{\phi_1\}$ must be contained in kernel classes. Thus, the possible kernel classes of $\Omega$ are $\au(\sn)\cup E$ and $A\cup\{\phi_1\}$ or are $\au(\sn)$ and $\sen(\sn)$.

    Conversely, it is easily verified that any choice for potential image and kernel corresponds to a rank 2 endomorphism.
\end{proof}
\subsection{Endomorphisms of low rank on the units}\label{subsec:trivonunit}
This subsection is dedicated to describing the endomorphisms of $\en(\sn)$ of rank at least $3$, which also have a low rank on the units. According to Lemma~\ref{lem:image_of_AutSn} these are endomorphisms which have either rank $1$ or $2$ on $\au(\sn)$. We start by considering the situation where the image of $\au(\sn)$ is a singleton.
\begin{theorem}\label{thm:endos_trivial_on_units}
	Let $\Phi$ be an endomorphism of $\en(\sn)$ such that $|\au(\sn)\Phi|=1$ and $rank(\Phi)\geq 3$. Then: 
	\begin{enumerate}[(i)]
		\item $\phi_1\Phi=\phi_1$;
		\item $\au(\sn)\Phi=\psi_1$; and
		\item for any $\phi_g\in E\cup A$, $(\phi_g\au(\sn))\Phi=\{\phi_k\}$ where $g$ and $k$ are both in $\an$ or both in $\sn\setminus\an$.
	\end{enumerate}
	
	Moreover, any map on $\en(\sn)$ satisfying the three conditions above is an endomorphism.
\end{theorem}
\begin{proof}
	Let us first assume that $\Phi$ is an endomorphism of $\en(\sn)$ with rank at least 3 that restricts to a constant map on the units. This means in particular that the rank of $\Phi$ on $\sen(\sn)$ is at least $2$. For this to happen, we must have $\phi_1\Phi=\phi_1$ by Lemma~\ref{ImPhi1}, so that $(i)$ holds.
	
	As $\Phi$ is a constant map on $\au(\sn)$ we must have $\psi_g\Phi=\psi_1\Phi$ for all $g\in\sn$. Suppose that $\Phi$ maps $\psi_1$ to $\phi_g\in E\cup\{\phi_1\}$. Then $\im(\Phi)\subseteq\sen(\sn)$ and for all $\phi_h\in E\cup A$ we have $\phi_h\Phi$ is equal to either $\phi_g$ or $\phi_1$, since
		\[\phi_h\Phi=(\phi_h\psi_1)\Phi=\phi_h\Phi\psi_1\Phi=\phi_t\phi_g=\begin{cases}
			\phi_g	&\text{if }t\in\sn\setminus\an\\
			\phi_1	&\text{if }t\in\an
		\end{cases}
.\]
	But then $\Phi$ is either rank 1 or 2, a contradiction. Thus $\psi_1\Phi=\psi_1$ and $(ii)$ holds.
	
	Finally, let $\phi_g\in E\cup A$ and $\phi_{g^h}\in\phi_g\au(\sn)$. Then we have that
		\[\phi_{g^h}\Phi = (\phi_g\psi_h)\Phi=(\phi_g\Phi)(\psi_h\Phi)=(\phi_g\Phi)\psi_1=\phi_g\Phi,\]
	so that $\phi_g\au(\sn)$ lies in a kernel class. Now, by the above remarks and Lemmas~\ref{Asubset} and~\ref{lem:phi1}, we must have that $A\Phi\subseteq A\cup\{\phi_1\}$ as well as $E\Phi\subseteq E$. Letting $\phi_k=\phi_g\Phi$, it then follows that $g$ and $k$ are either both even or both odd permutations giving (iii).

	Conversely, let $\Phi$ be a map on $\en(\sn)$ satisfying the three conditions of the statement. By $(iii)$, we have that $E\Phi\subseteq E$ and $A\Phi\subseteq A\cup\{\phi_1\}$, which combined with $(i)$ gives us that the restriction of $\Phi$ to $\sen(\sn)$ is an endomorphism by Proposition~\ref{prop:SingPrev}. All that remains to prove that $\Phi$ is an endomorphism is to consider products involving a singular element with a unit. So let $\psi_g\in\au(\sn)$ and $\phi_h\in\sen(\sn)$. Recalling that $\psi_g\Phi=\psi_1$ by $(ii)$ and that $\phi_{h^g}\Phi=\phi_h\Phi$ by $(iii)$, we get that
		\[(\psi_g\phi_h)\Phi=\phi_h\Phi=\psi_1(\phi_h\Phi)=\psi_g\Phi\phi_h\Phi,\]
	and
		\[(\phi_h\psi_g)\Phi=\phi_{h^g}\Phi=\phi_h\Phi=(\phi_h\Phi)\psi_1=\phi_h\Phi\psi_g\Phi,\]
	which finishes to show that $\Phi$ is an endomorphism.
\end{proof}
Since Theorem~\ref{thm:endos_trivial_on_units} deals with all the possibilities of endomorphisms which have rank $1$ on the group of units $\au(\sn)$, we now move towards those that have rank $2$ on $\au(\sn)$, and thus have global rank at least $3$.
\begin{theorem}\label{thm:rank3}
	Let $\Phi$ be an endomorphism of $\en(\sn)$ such that $|\au(\sn)\Phi|=2$ and $rank(\Phi)=3$. Then $\Phi$ is defined by
	\[\Phi:\omega\mapsto\begin{cases}
				\psi_g	&\text{if }\omega=\psi_s\text{ and }s\in \sn\setminus\an,\\
				\psi_1	&\text{if }\omega=\psi_s\text{ and }s\in\an,\\
				\phi_k	&\text{otherwise},
			\end{cases}\]
	where $g$ has order 2, and $\phi_k\in E\cup\{\phi_1\}$ with $k^g=k$.
	
	Moreover, any map of $\en(\sn)$ satisfying these conditions is an endomorphism.
\end{theorem}
\begin{proof}
	By Lemma~\ref{lem:image_of_AutSn} we know that $\au(\sn)\Phi=\{\psi_g,\psi_1\}$ for some $g\in \sn\setminus\{1\}$ of order $2$.
    We now show that $\sen(\sn)\Phi=\{\phi_k\}$ for some $\phi_k\in E\cup\{\phi_1\}$.

   Suppose for contradiction that $\phi_t\Phi\in\au(\sn)$ for some $\phi_t\in\sen(\sn)$. Then for all $\phi_g\in E$, we have $\phi_g\Phi\phi_t\Phi = (\phi_g\phi_t)\Phi = \phi_t \Phi$, so that $\phi_g\Phi$ must be an idempotent in $\au(\sn)$. Consequently, $\phi_g\Phi=\psi_1$, but then for all automorphisms $\psi_s$ we have $\psi_s\Phi = \psi_s \Phi \psi_1 = \psi_s\Phi \phi_g\Phi = (\psi_s \phi_g)\Phi= \phi_g \Phi  = \psi_1$, contradicting $|\au(\sn)|=2$. Since $\Phi$ has rank 3, and idempotents must be sent to idempotents, it follows that $\sen(\sn)\Phi=\{\phi_k\}$ for some $\phi_k\in E \cup \{\phi_1\}$. Additionally, for $h\in\sn\setminus\an$ we have that $\phi_k=\phi_k\Phi=(\phi_k\psi_h)\Phi=(\phi_k\Phi)(\psi_h\Phi)=\phi_k\psi_g=\phi_{k^g}$, which shows that $k^g=k$.

	It is easy to verify that any such map is an endomorphism.
\end{proof}
\begin{lemma}\label{lem:partpres}
	Let $\Phi$ be an endomorphism of $\en(\sn)$ such that $|\au(\sn)\Phi|=2$ and $\rank(\Phi)\geq 4$. Then the following must hold:
	\begin{enumerate}[(i)]
		\item $\au(\sn)\Phi\subseteq\au(\sn)$,
		\item $E\Phi\subseteq E$,
		\item $A\Phi\subseteq A\cup\{\phi_1\}$,
		\item $\phi_1\Phi=\phi_1$,
		\item for any $\phi_k\in\sen(\sn)$, we have that $|(\phi_k\au(\sn))\Phi|=1$.
	\end{enumerate}
	Moreover, letting $\au(\sn)\Phi=\{\psi_g,\psi_1\}$ yields $g\in C(h)$ for all $\phi_h\in\im(\Phi)$.
\end{lemma}
\begin{proof}
	Given that $|\au(\sn)\Phi|=2$, Lemma~\ref{lem:image_of_AutSn} yields that $\au(\sn)\Phi=\{\psi_g,\psi_1\}$ for some $g\in\sn$ of order two, with $\psi_s\Phi=\psi_1$ precisely if $s\in\an$. Thus, $\au(\sn)\Phi\subseteq\au(\sn)$. Moreover, Lemma~\ref{ImPhi1} yields that $\phi_1\Phi=\phi_1$, since otherwise, $\Phi$ would have rank at most $3$. The fact that $A\Phi\subseteq A\cup\{\phi_1\}$ and $E\Phi\subseteq E$ follows directly from Lemmas~\ref{Asubset} and~\ref{lem:phi1}, since we would otherwise have a contradiction on the rank of $\Phi$. This proves $(i)$-$(iv)$. Since $(v)$ clearly holds for $k=1$, we assume for the rest of the proof that $k\neq 1$.

Now, let $\phi_k\in\sen(\sn)$ and suppose that $\phi_k\Phi=\phi_h$. Then, for all $s\in\sn$, we have that
	\[\phi_{k^s}\Phi=(\phi_{k}\psi_s)\Phi=\phi_k\Phi\psi_s\Phi=
		\begin{cases}
			\phi_{h^g}	&\text{if }s\in\sn\setminus\an,\\
			\phi_{h}	&\text{if }s\in\an.
		\end{cases}\]
    In particular, if we take $s$ to be a transposition that is a factor of $k$, we get that $s\in C(k)\cap (\sn\setminus\an)$ and $k^s=k$. Hence,
		\[\phi_h=\phi_k\Phi=\phi_{k^s}\Phi=(\phi_k\psi_s)\Phi=(\phi_k\Phi)(\psi_s\Phi)=\phi_h\psi_g=\phi_{h^g},\]
	which shows that $h=h^g$.
	Therefore, $(\phi_k\au(\sn))\Phi=\{\phi_h\}$ and $g\in C(h)$ for all $\phi_h\in\im(\Phi)$.
\end{proof}
\begin{definition}\label{def:Phi}
Let $g\in\sn$ be an element of order $2$, and $S$ be a map of $\sen(\sn)$ that satisfies each of the following:
    \begin{enumerate}[(1)]
        \item $ES\subseteq E$,
        \item $AS\subseteq A\cup\{\phi_1\}$,
        \item $\phi_1S=\phi_1$,
        \item each kernel class of $S$ is a union of conjugacy classes,
        \item for all $\phi_h$ in $\im(S)$ we have that $h^g=h$.
    \end{enumerate}
Then we define $\Phi^g_S:\en(\sn)\rightarrow\en(\sn)$ as follows:
        \[\Phi^g_S:\omega\mapsto\begin{cases}
            \psi_g  &\text{if }\omega=\psi_k,\;k\in\sn\setminus\an,\\
            \psi_1  &\text{if }\omega=\psi_k,\;k\in\an,\\
            \omega S&\text{if }\omega\in\sen(\sn).
        \end{cases}\]
\end{definition}
\begin{theorem}\label{thm:rank2onunits}
	Each $\Phi^g_S$ as defined in Definition \ref{def:Phi} is an endomorphism of $\en(\sn)$. Further, any endomorphism of $\en(\sn)$ with $|\au(\sn)\Phi|=2$ and $\rank(\Phi)\geq 4$ is a map of the form $\Phi^g_S$.
\end{theorem}
\begin{proof}
    Given a map $\Phi^g_S$ as above, properties $(i)$-$(iii)$ of the map $S$ mean that Proposition~\ref{prop:SingPrev} applies, and we immediately get that the composition of elements in $\sen(\sn)$ is preserved. Furthermore, since $g$ has order $2$, it is clear that composition within $\au(\sn)$ are also preserved. Since elements of $\au(\sn)$ are left identities for those in $\sen(\sn)$, we get that all compositions of the form $\psi_p\phi_q$ are preserved by $\Phi^g_S$. For compositions of the form $\phi_q\psi_p$, note that by property $(iv)$ we have $(\phi_q\psi_p)\Phi=(\phi_{q^p})\Phi=\phi_q\Phi$. The fact that these compositions are also preserved by $\Phi^g_S$ the relies on the fact (guaranteed by property $(v)$) that $g\in C(h)$ where $\phi_h=\phi_q S$. Thus, $\Phi^g_S$ is an endomorphism of $\en(\sn)$.

    Now, let $\Phi$ be an endomorphism of $\en(\sn)$ such that $|\au(\sn)\Phi|=2$ and $\rank(\Phi)\geq 4$. Then, by Lemma~\ref{lem:image_of_AutSn}, we have that $\au(\sn)\Phi=\{\psi_g,\psi_1\}$ for some $g\neq 1$ of order $2$. By Lemma~\ref{lem:partpres} we have that $E\Phi\subseteq E$, $A\Phi\subseteq A\cup\{\phi_1\}$, and $\phi_1\Phi=\phi_1$. Finally, by Lemma~\ref{lem:partpres}, we have that $g$ is in $C(k)$ for all $\phi_k\in\im(S)$ and that $|(\phi_k\au(\sn))\Phi|=1$ so that every conjugacy class in $\sen(\sn)$ is contained in a kernel class of $\Phi$. Thus, by setting $\phi_k S=\phi_k\Phi$, we get that $S$ satisfies all the conditions of Definition~\ref{def:Phi} so that $\Phi=\Phi^g_S$.
\end{proof}
\subsection{Endomorphisms of full rank on the units}\label{subsec:injectonunit}
We now consider the last case of Lemma~\ref{lem:image_of_AutSn}, which is when an endomorphism $\Delta\in\en_2(\sn)$ restricts to a bijection of $\au(\sn)$. This means in particular that there exists $h\in\sn$ such that $\psi_s\Delta=\psi_s\Psi_h$ for all $\psi_s\in\au(\sn)$. If we then let $\Delta':=\Delta\Psi_{h^{-1}}$, we have that $\Delta'$ is an endomorphism of $\en(\sn)$ that restricts to the identity on $\au(\sn)$.

Conversely, we can recover $\Delta$ as the product $\Delta'\Psi_h$. That is, any endomorphism of $\en(\sn)$ that acts bijectively on $\au(\sn)$ is the product of an endomorphism that restricts to the identity on $\au(\sn)$, together with some automorphism $\Psi_h$. Thus, we focus our attention on endomorphisms $\Delta$ that restrict to the identity map on $\au(\sn)$. Our aim in this section is to prove Theorem \ref{thm:endo-identityOnUnits}, which exactly characterises all such endomorphisms. To aid the proof, we first establish several small lemmas.
\begin{lemma}\label{lem:inj-SEndKernel}
	Let $\Delta$ be an endomorphism of $\en(\sn)$ such that $\Delta$ restricts to the identity map on $\au(\sn)$. If $\sen(\sn)$ is a kernel class of $\Delta$, then $\phi_g\Delta=\phi_1$ for all $\phi_g\in\sen(\sn)$.
\end{lemma}
\begin{proof}
Since $\sen(\sn)$ is a kernel class, the map $\Delta$ must send all singular endomorphisms to $\phi_1\Delta$. Moreover, by Lemma~\ref{ImPhi1}, we have that $\phi_1\Delta=\phi_k$ for some idempotent $\phi_k\in E\cup\{\phi_1\}$.
However, $\im(\Delta)$ is a subsemigroup of $\en(\sn)$ which contains $\au(\sn)$ so we must have that $\phi_{k^h}=\phi_k\psi_h=\phi_k$ for all $h\in\sn$. Thus (by definition) we see that $k^h=k$ for all $h\in\sn$, but since $\sn$ is centreless this forces $\phi_1\Delta=\phi_1$, as required.
\end{proof}
Having eliminated the case where we have a single kernel class outside of $\au(\sn)$, we now determine the remaining possible kernel classes that $\Delta$ can have.
\begin{lemma}\label{lem:injprespart}
	Let $\Delta$ be an endomorphism of $\en(\sn)$ such that $\Delta$ restricts to the identity map on $\au(\sn)$. If $\sen(\sn)$ is not a kernel class of $\Delta$, then $\Delta$ must satisfy the following conditions:
\begin{enumerate}[(i)]
	\item $E\Delta\subseteq E$,
	\item $A\Delta\subseteq A\cup\{\phi_1\}$,
	\item $\phi_1\Delta=\phi_1$.
\end{enumerate}
\end{lemma}
\begin{proof}
    We know by Lemma~\ref{ImPhi1} that if $\phi_1\Delta\neq\phi_1$ then either $\Delta$ cannot be injective on $\au(\sn)$ or $\sen(\sn)$ must be a kernel class of $\Delta$, which contradicts our assumptions. Hence $(iii)$ must hold. The other two conditions come from the contrapositive of Lemma~\ref{lem:phi1}, which gives us that $E\Delta\subseteq E$, and of Lemma~\ref{Asubset} which tells us that $A\Delta\subseteq A\cup\{\phi_1\}$.
\end{proof}
Having established the action on $\au(\sn)\cup\{\phi_1\}$, we now analyse how $\Delta$ acts on the remaining sets (namely, $E$ and $A$) in the partition of $\en(\sn)$.
\begin{lemma}\label{lem:injcentsame}
	Let $\Delta$ be an endomorphism of $\en(\sn)$ such that $\Delta$ restricts to the identity map on $\au(\sn)$ and let $\phi_g\in E\cup A$. If $k\in\sn$ is such that $\phi_g\Delta=\phi_k$, then $C(g)\leq C(k)$. Further, if $k\neq 1$ and $k$ is not a transposition, then $k=g$.
\end{lemma}
\begin{proof}
Let $\phi_g\in E\cup A$ and suppose that $\phi_g\Delta=\phi_k$. Take $h\in C(g)$, so that $\phi_g\psi_h=\phi_g$. Then we have
	 \[\phi_{k^h}=\phi_k\psi_h=\phi_g\Delta\psi_h\Delta=(\phi_g\psi_h)\Delta=\phi_g\Delta=\phi_k,\]
which by definition forces $k^h=k$ that is, $h\in C(k)$. Using Lemma \ref{lem:injprespart}, we have that $g$ and $k$ have the same parity, and so Lemma \ref{lem:gktrans} applies to give $k=g$.
\end{proof}
\begin{lemma}\label{lem:Delta-identity-conjugacyClass}
	Let $\Delta$ be an endomorphism of $\en(\sn)$ such that $\Delta$ restricts to the identity map on $\au(\sn)$ and let $\phi_g\in E\cup A$. If $\phi_g\Delta\in\phi_g\au(\sn)$, then $\Delta$ restricts to the identity map on $\phi_g\au(\sn)$.
\end{lemma}
\begin{proof}
Since $\phi_g\Delta$ is in $\phi_g\au(\sn)$ we have that $\phi_g\Delta=\phi_k$ where $k={g^s}$ for some $s\in\sn$. As $\Delta$ restricts to the identity map on $\au(\sn)$, for all $t\in\sn$ we also get that $\phi_{g^t}=\phi_{g^s}\psi_{s^{-1}t}=(\phi_g\psi_{s^{-1}t})\Delta$, so that $(\phi_g\au(\sn))\Delta=\phi_g\au(\sn)$. Now consider the sequence of elements $\phi_g\Delta^i$. By our previous observation, each lies in the (finite) conjugacy class $\phi_g\au(\sn)$.
Thus $\phi_g\Delta^m=\phi_g$ for some $m$. Writing $\phi_g\Delta^i=\phi_{k_i}$, Lemma~\ref{lem:injcentsame} applies to give:
	\[C(g)\leq C(k)\leq C(k_2)\leq\cdots\leq C(k_m)\leq C(g),\]
which forces $C(g)=C(k)$. Since $g$ and $k$ are conjugate elements of $\sn$ of order two, we get that $g=k$ by Corollary~\ref{lem:SameCentPart2}. Thus, $\Delta$ restricts to the identity map on $\phi_g\au(\sn)$ as required.
\end{proof}
We next show that each conjugacy class of $A$ is either fixed by $\Delta$ or is sent to $\phi_1$.
\begin{lemma}\label{lem:deltaonA}
	Let $\Delta$ be an endomorphism of $\en(\sn)$ such that $\Delta$ restricts to the identity map on $\au(\sn)$. If $\sen(\sn)$ is not a kernel class of $\Delta$, then for each $\phi_g\in A$ one of the following happens:
	\begin{enumerate}[(i)]
		\item $\phi_h\Delta=\phi_h$ for all $\phi_h\in\phi_g\au(\sn)$; or
		\item $\phi_h\Delta=\phi_1$ for all $\phi_h\in\phi_g\au(\sn)$.
	\end{enumerate}
\end{lemma}
\begin{proof}
    Let $\phi_g\in A$. By Lemma~\ref{lem:injprespart} we know that $\phi_g\Delta\in A\cup\{\phi_1\}$. If $\phi_g\Delta=\phi_k\in A$, then $k$ is not a transposition, which forces $k=g$ by Lemma~\ref{lem:injcentsame}. By Lemma~\ref{lem:Delta-identity-conjugacyClass}, we then get that $\Delta$ restricts to the identity on $\phi_g\au(\sn)$.

    On the other hand, if $\phi_g\Delta=\phi_1$, then for each $\phi_{g^s}\in\phi_g\au(\sn)$ we get that $\phi_{g^s}\Delta=(\phi_g\psi_s)\Delta=(\phi_g\Delta)\psi_s=\phi_1$ so that $(\phi_g\au(\sn))\Delta=\{\phi_1\}$ in this case.
\end{proof}
Since elements of the form $\phi_k$ with $k$ a transposition only appear in $E$, the situation for this set is slightly different, and the following lemma gives us that either $\Delta$ fixes all of its elements, or it only moves those with a fix of size two.
\begin{lemma}\label{lem:inj-DeltaonE}
	Let $\Delta$ be an endomorphism of $\en(\sn)$ such that $\Delta$ restricts to the identity map on $\au(\sn)$. If $\sen(\sn)$ is not a kernel class of $\Delta$, then for each $\phi_g\in E$ we either have
	\begin{enumerate}[(i)]
		\item $\phi_g\Delta=\phi_g$; or
		\item \label{item:lem:DeltaonE-gto(ij)} $\phi_g\Delta=\phi_{(i\;j)}$, where $\{i,j\}=\Fix(g)$, and this can only occur whenever $n$ is divisible by $4$.
	\end{enumerate}

    Moreover, $\phi_g\Delta=\phi_{(i\;j)}$ if and only if $\phi_h\Delta\neq\phi_h$ for some (all) $\phi_h\in E$ such that $|\Fix(h)|=2$.
\end{lemma}
\begin{proof}
    By Lemma~\ref{lem:injprespart}, we have that for all $\phi_g\in E$ there exists $\phi_k\in E$ such that $\phi_g\Delta=\phi_k$. By Lemma~\ref{lem:injcentsame}, we have that $C(g)\leq C(k)$, and if $k$ is not a transposition, then $g=k$. On the other hand, if $k$ is a transposition, then since $C(g)\leq C(k)$  Lemma~\ref{lem:gh-subgpcent-hIsTransp-equalOrActFully} yields either $g=k$, or $n$ is divisible by $4$ and we also have that $g$ and $k$ are disjoint elements with $\Fix(gk)=\emptyset$. This means in particular that in the latter case $\supp(k)=\Fix(g)=[1,n]\setminus\supp(g)$, so that the transposition $k$ is uniquely determined by $g$.

    For the second part of the Lemma, notice that when $n$ is divisible by $4$, all elements of order two in $\sn$ which have a fix of size $2$ can be written as products of $\frac{n-2}{2}$ disjoint transpositions and are thus conjugate. Therefore, if we let $\Fix(g)=\{i,j\}$ and $\phi_h\in E$ is such that $\Fix(h)=\{p,q\}$, we have that $h=g^s$ for some $s\in\sn$. Moreover, $(is)h=igs=is$, which means that $is\in\{p,q\}$, and similarly for $js$. Hence, we must have that $(i\;j)^s=(p\;q)$. Now if we assume that $\phi_g\Delta=\phi_{(i\;j)}$, since $n\geq 7$ we get that
    	\[\phi_h\Delta=(\phi_g\psi_s)\Delta=\phi_g\Delta\psi_s\Delta=\phi_{(i\;j)}\psi_s=\phi_{(i\;j)^s}=\phi_{(p\;q)}\neq\phi_h.\]
    Similarly, if we know that $\phi_h\Delta\neq\phi_h$, then by the first part we must have that $\phi_h\Delta=\phi_{(p\;q)}$, and using a similar computation with $g=h^{s^{-1}}$, we get that $\phi_g\Delta=\phi_{(p\;q)^{s^{-1}}}=\phi_{(i\;j)}$, as required.
\end{proof}
We now have enough information to build all the endomorphisms of $\en(\sn)$ that restrict to the identity on $\au(\sn)$.
\begin{theorem}\label{thm:endo-identityOnUnits}
	Let $\Delta$ be an endomorphism of $\en(\sn)$ such that $\Delta$ restricts to the identity map on $\au(\sn)$. Then the map $\Delta$ must have one of the following form.
	\begin{enumerate}[(i)]
	\item \label{item:thm:inj-elmt-singkernel} When $\sen(\sn)$ is a kernel class of $\Delta$, then we have
		\[\Delta:\omega\mapsto\begin{cases}
		\psi_g	&\text{if }\omega=\psi_g,\\
		\phi_1	&\text{otherwise.}
	\end{cases}\]
	\item \label{item:thm:inj-elmt-Efixed} When $\sen(\sn)$ is not a kernel class of $\Delta$ and all the elements of $E$ are fixed, then
		\[\Delta:\omega\mapsto\begin{cases}
		\psi_g	&\text{if }\omega=\psi_g,\\
		\phi_g	&\text{if }\omega=\phi_g\in E\cup (A\setminus X),\\
		\phi_1	&\text{if }\omega=\phi_g\in X,
	\end{cases}\]
	where $X$ is a union of conjugacy classes in $A\cup\{\phi_1\}$ which always includes $\phi_1$.
	\item \label{item:thm:inj-elmt-Esplit} When $\sen(\sn)$ is not a kernel class of $\Delta$, $n$ is divisible by $4$ and some elements of $E$ are not fixed, then
		\[\Delta:\omega\mapsto\begin{cases}
		\psi_g			&\text{if }\omega=\psi_g,\\
		\phi_{(i\;j)}	&\text{if }\omega=\phi_g\in F\text{ and }\Fix(g)=\{i,j\},\\
		\phi_g			&\text{if }\omega=\phi_g\in (E\setminus F)\cup(A\setminus X),\\
		\phi_1			&\text{if }\omega=\phi_g\in X,
	\end{cases}\]
	where $F=\{\phi_g\in E\mid |\Fix(g)|=2\}$ and $X$ is a union of conjugacy classes in $A\cup\{\phi_1\}$ which always contains $\phi_1$.
	\end{enumerate}

    Moreover, for any set $X\subseteq A\cup \{\phi_1\}$ that contains $\phi_1$ and is a union of conjugacy classes, any map of the form \ref{item:thm:inj-elmt-Efixed} and~\ref{item:thm:inj-elmt-Esplit} (whenever $n$ is divisible by $4$) created from $X$ is an endomorphism of $\en(\sn)$ which restricts to the identity map on $\au(\sn)$.
\end{theorem}
\begin{proof}
    Clearly, since $\Delta$ restricts to the identity map on $\au(\sn)$ we have that $\psi_g\Delta=\psi_g$ for all $\psi_g\in\au(\sn)$. Additionally, if $\sen(\sn)$ is a kernel class of $\Delta$, then we are in the situation of Lemma~\ref{lem:inj-SEndKernel} and $\Delta$ is exactly described by part~\ref{item:thm:inj-elmt-singkernel}.

    So we now assume that $\sen(\sn)$ is not a kernel class of $\Delta$. By Lemma~\ref{lem:injprespart}, we know that $E\Delta\subseteq E$, $A\Delta\subseteq A\cup\{\phi_1\}$ and $\phi_1\Delta=\phi_1$. Let $X:=\{\phi_g\in\en(\sn)\mid\phi_g\Delta=\phi_1\}$, so that $X\subseteq A\cup\{\phi_1\}$ by definition. Moreover, by Lemma~\ref{lem:deltaonA} we have that $\phi_g\in X$ if and only if $\phi_g\au(\sn)\subseteq X$, so that $X$ is a union of conjugacy classes. Additionally, for all $\phi_g\in A\setminus X$, the same lemma gives us that $\phi_g\Delta=\phi_g$. Hence, the action of $\Delta$ on the elements of $A\cup\{\phi_1\}$ corresponds to the (same) description given on the elements of this set in parts~\ref{item:thm:inj-elmt-Efixed} and~\ref{item:thm:inj-elmt-Esplit}.

    All that is left to do is to understand how $\Delta$ acts on $E$. Suppose first that $n$ is not divisible by $4$. Then, by Lemma~\ref{lem:inj-DeltaonE}, we must have that $\phi_g\Delta=\phi_g$ for all $\phi_g\in E$, and we get the map described in part~\ref{item:thm:inj-elmt-Efixed}. On the other hand, if $n$ is divisible by $4$, then the set $F:=\{\phi_g\in E\mid |\Fix(g)|=2\}$ is non-empty. By Lemma~\ref{lem:inj-DeltaonE} once again, either $\phi_g\Delta=\phi_g$ for all $\phi_g\in F$ (and thus also for the whole of $E$), and we get the map described in part~\ref{item:thm:inj-elmt-Efixed}, or there exists $\phi_h\in F$ such that $\phi_h\Delta \neq \phi_h$. In this latter case, we also know that for all $\phi_g\in F$ we must have $\phi_g\Delta\neq\phi_g$ and $\phi_g\Delta=\phi_{(i\;j)}$ where $\{i,j\}=\Fix(g)$. Since $\Delta$ also fixes all elements of $E\setminus F$, it follows that $\Delta$ is then as given in part~\ref{item:thm:inj-elmt-Esplit}.

    Conversely, let $X \subseteq A\cup\{\phi_1\}$ be a set containing $\phi_1$ and consisting of a union of conjugacy classes. Let $\Delta$ be a map described in parts~\ref{item:thm:inj-elmt-Efixed} or~\ref{item:thm:inj-elmt-Esplit} by using $X$. We aim to show that this is an endomorphism. It is immediate to see that $\au(\sn)\Delta=\au(\sn)$, $E\Delta\subseteq E$, $A\Delta\subseteq A\cup\{\phi_1\}$ and $\phi_1\Delta=\phi_1$. Therefore, by Proposition~\ref{prop:SingPrev}, we have that $(\phi_g\phi_h)\Delta=\phi_g\Delta\phi_h\Delta$ for all $\phi_g,\phi_h\in\sen(\sn)$. Moreover, all products only involving elements of $\au(\sn)$ or of the form $\psi_s\phi_g$ are trivially verified, while those of the form $\phi_g\psi_s$ can be easily checked whenever $\phi_g\Delta=\phi_g$ or $\phi_g\Delta=\phi_1$. The only remaining product to verify is when $\Delta$ is of type~\ref{item:thm:inj-elmt-Esplit} in the case where $\phi_g\in F$ and $\phi_g\Delta=\phi_{(i\;j)}$ with $\{i,j\}=\Fix(g)$. So let $\psi_s\in\au(\sn)$ and define $h=g^s$. Then $h$ is such that $\Fix(h)=\{is,js\}$ and by definition of our map $\Delta$ we get that
    \[(\phi_g\psi_s)\Delta=\phi_h\Delta=\phi_{(is\;js)}=\phi_{(i\;j)^s}=\phi_{(i\;j)}\psi_s=\phi_g\Delta\psi_s\Delta.\]
Therefore, $\Delta$ is an endomorphism.
\end{proof}
Combining the previous theorem with the remarks at the beginning of this section, we immediately get the following corollary.
\begin{corollary}
    Let $\Delta$ be an endomorphism of $\en(\sn)$ such that $\Delta$ restricts to a bijection on $\au(\sn)$. Then $\Delta=\Delta'\Psi_g$ for some $g\in\sn$ where $\Delta'$ is one of the map described in Theorem~\ref{thm:endo-identityOnUnits}.
\end{corollary}
\subsection{Summary}
We conclude this section by providing a complete list of the elements of $\en_2(\sn)$. This notation will be heavily used in the next section where we determine the automorphism group of $\en_2(\sn)$.

In all that follows  $\omega\in\en(\sn)$. We list the automorphisms of $\en(\sn)$ followed by the singular endomorphisms, grouped by common features.
\begin{enumerate}[(i)]
	\item The automorphisms:
		\[\Psi_g:\omega\mapsto\psi_g^{-1}\omega\psi_g\text{ where }g\in\sn.\]
	\item The endomorphisms that are bijective on $\au(\sn)$ and collapse two conjugacy classes of $E$ into one (these exist precisely when $n$ is divisible by $4$):
		\[\Xi^g_X:\omega\mapsto\begin{cases}
			\phi_{(i\;j)^g}			&\text{if }\omega=\phi_t\in E\text{ and }\Fix(t)=\{i,j\},\\
			\phi_1					&\text{if }\omega\in X,\\
			\psi_g^{-1}\omega\psi_g	&\text{otherwise},
		\end{cases}\]
		where $g\in\sn$ and $X$ is either the empty set or is a union of conjugacy classes in $A\cup\{\phi_1\}$ that must always contain $\phi_1$.
	\item The endomorphisms that are a bijection on $\au(\sn)$ and either a bijection on $E$ or send $E$ to $\{\phi_1\}$:
		\[\Delta^g_X:\omega\mapsto\begin{cases}
			\phi_1					&\text{if }\omega\in X,\\
			\psi_g^{-1}\omega\psi_g	&\text{otherwise},
		\end{cases}\]
		where $g$ is in $\sn$ and $X$ is either the set $\sen(\sn)$ or is a non-empty union of conjugacy classes in $A\cup\{\phi_1\}$ that always contains $\phi_1$ and at least one conjugacy class of $A$.
	\item The endomorphisms that restrict to singular endomorphisms of $\au(\sn)$:
		\[\Phi_S^g:\omega\mapsto\begin{cases}
        \psi_g	&\text{if }\omega=\psi_h\text{ for }h\in\sn\setminus\an,\\
        \psi_1	&\text{if }\omega=\psi_h\text{ for }h\in\an,\\
        \phi_kS	&\text{if }\omega=\phi_k\in\sen(\sn),
    \end{cases}\]
	where $g$ is an element of $\sn$ such that $g^2=1$ and $S$ is a map satisfying one of the following:
    \begin{enumerate}
    	\item\label{PhiMapZ} $\im(S)=\{\phi_t\}$ for some $\phi_t$ in $E\cup\{\phi_1\}$ such that $t^g=t$; or
		\item $ES\subseteq E$, $AS\subseteq A\cup\{\phi_1\}$, $\phi_1S=\phi_1$, the kernel classes of $S$ are unions of conjugacy classes, and for all $\phi_t$ in $\im(S)$ we have $t^g=t$.
    \end{enumerate}
    In the special case of \ref{PhiMapZ} where the image contains only $\phi_1$, we will denote the map $S$ by $Z$. 
    \item The rank $2$ endomorphisms with kernel class $\au(\sn)\cup E$ mapping to the identity:
    	\[\Omega_{1,g}:\omega\mapsto\begin{cases}
                \psi_1&\text{if }\omega\in\au(\sn)\cup E,\\
                \phi_g&\text{if }\omega\in A\cup\{\phi_1\},
            \end{cases}\]
       	where $\phi_g$ is in $E\cup\{\phi_1\}$.
   	\item The rank $2$ endomorphisms with kernel class $\au(\sn)\cup E$ mapping to an idempotent other than the identity:
		\[\Omega_{2,g}:\omega\mapsto\begin{cases}
			\phi_g	&\text{if }\omega\in\au(\sn)\cup E,\\
			\phi_1	&\text{if }\omega\in A\cup\{\phi_1\},
			\end{cases}\]
		where $\phi_g$ is in $E$.
	\item The rank $2$ endomorphisms with kernel class $\sen(\sn)$ mapping to $\phi_1$:
		\[\Omega_{3,g}:\omega\mapsto\begin{cases}
		\phi_g	&\text{if }\omega\in\au(\sn),\\
		\phi_1	&\text{if }\omega\in\sen(\sn),
	\end{cases}\]
	where $\phi_g$ is in $E$.
   \item The constant maps to an idempotent $\varepsilon$, denoted $\Gamma_\varepsilon$.
\end{enumerate}
\begin{remark}
In the above table, rank $2$ endomorphisms with kernel class $\sen(\sn)$ mapping to an idempotent $\phi_g$ are the elements of the form $\Phi^1_S$ where $S$ maps $\sen(\sn)$ to $\phi_g$. One can  easily verify using the previous results of this section that no duplicate appears in this summary, and that any choice of $g\in\sn$, $X\subseteq \sen(\sn)$ or map $S:\sen(\sn)\to\sen(\sn)$ satisfying the conditions given in each case gives rise to an endomorphism.
\end{remark}
\begin{remark}\label{rem:Gamma_min_ideal}
It is clear that the set $\Gamma:=\{\Gamma_\varepsilon\mid\varepsilon\in\en(\sn)\text{ is idempotent}\}$ is the minimal ideal of $\en_2(\sn)$ as it contains all rank 1 maps. Moreover, from Example~\ref{ex:ch}, we get that this is also a characteristic subset of $\en_2(\sn)$.
\end{remark}
\section{Endomorphisms of \texorpdfstring{$\en(\sn)$}{} - the monoid \texorpdfstring{$\en_2(\sn)$}{}}
\label{sec:endomonoid}
The following is the composition table of $\en_2(\sn)$. We abuse notation on the $\Omega$ elements and (if $S$ is a map) write $\Omega_{i,gS}$ for $\Omega_{i,h}$ where $\phi_gS=\phi_h$. Additionally, we define $S(g)$ to be the constant map of $\sen(\sn)$ that has image $\{\phi_g\}$ and in the case where $g=1$ we put $Z=S(1)$. We then define $X^\flat $ and $X^\sharp$ to be the maps ${\Xi_X\vert}_{\sen(\sn)}$ and ${\Delta_X\vert}_{\sen(\sn)}$ respectively. Again abusing notation, in the case of $\Delta_X^g$ where $X=\sen(\sn)$, so that $X^\sharp=Z$, we will denote this element by $\Delta^g_Z$. These conventions reflect the fact that $\Delta^g_Z$ and $\Phi^g_Z$ restrict to the same map of $\sen(\sn)$.
	\begin{table}[H]
	\[\setlength{\extrarowheight}{2mm}
    \begin{array}{l|*{9}{l}}
        		  & \Psi_g   			 & \Xi^g_X			    	   	& \Delta^g_X   			 	  & \Phi^g_S 			   & \Omega_{1,g} 				& \Omega_{2,g}			& \Omega_{3,g} 	& \Gamma_\omega\\
    \hline
    \Psi_h  	  & \Psi_{hg} 		     & \Xi^{hg}_X		    	   	& \Delta^{hg}_X 			  & \Phi^g_S	 		   & \Omega_{1,g} 				& \Omega_{2,g}			& \Omega_{3,g}	& \Gamma_\omega\\
    \Xi^h_Y  	  & \Xi^{hg}_Y		     & \Xi^{hg}_{X\cup Y}		   	& (\ref{XiDelta})		 	  & \Phi^g_{Y^\flat S}	   & \Omega_{1,g} 				& \Omega_{2,g}			& \Omega_{3,g}	& \Gamma_\omega\\
    \Delta^h_Y 	  & \Delta^{hg}_Y 	     & (\ref{DeltaXi})			   	& \Delta^{hg}_{Y\cup X}	  	  & \Phi^g_{Y^\sharp S}    & (\ref{DeltaOmega1}) 		& (\ref{DeltaOmega2})	& \Omega_{3,g}	& \Gamma_\omega\\
    \Phi^h_{T} 	  & \Phi^{h^g}_{T\Psi_g} & \Phi^{h^g}_{TX^\flat \Psi_g}	& \Phi^{h^g}_{TX^\sharp\Psi_g}& (\ref{PhiPhi})		   & (\ref{PhiOmega1})			& (\ref{PhiOmega2})		& \Omega_{3,g}	& \Gamma_\omega\\
    \Omega_{1,h}  & \Omega_{1,h^g}	     & (\ref{Omega1Xi})		  	   	& \Omega_{1,h^g}			  & \Omega_{1,hS}		   & (\ref{Omega1Omega1})		& (\ref{Omega1Omega2})	& \Omega_{2,g}	& \Gamma_\omega\\
    \Omega_{2,h}  & \Omega_{2,h^g}	     & (\ref{Omega2Xi})		  	   	& (\ref{Omega2Delta})		  & (\ref{Omega2Phi})	   & \Omega_{1,g} 				& \Omega_{2,g}   		&\Gamma_{\phi_1}& \Gamma_\omega\\
    \Omega_{3,h}  & \Omega_{3,h^g}	     & (\ref{Omega3Xi})		  	   	& (\ref{Omega3Delta})		  & (\ref{Omega3Phi})	   & \Phi^1_{S(g)}	   			& \Omega_{3,g}  		&\Gamma_{\phi_1}& \Gamma_\omega\\
    \Gamma_\lambda&\Gamma_{\lambda\Psi_g}&\Gamma_{\lambda\Xi^g_X} 		&\Gamma_{\lambda\Delta^g_X}   &\Gamma_{\lambda\Phi^g_S}&\Gamma_{\lambda\Omega_{1,g}}&\Gamma_{\lambda\Omega_{2,g}}&\Gamma_{\lambda\Omega_{3,g}}&\Gamma_\omega\\
    \end{array}\]
    \caption{Table for $\en_2(\sn)$}\label{tab:table}
    \end{table}
The cases $(1)$ to $(16)$ are as follows:
\begin{multicols}{2}
\begin{enumerate}
	\item \label{XiDelta}		$\begin{cases}
		\Xi^{hg}_{Y\cup X}		&\text{if }X\neq\sen(\sn)\\
		\Delta^{hg}_Z			&\text{if }X=\sen(\sn),
		\end{cases}$
	\item \label{DeltaXi}		$\begin{cases}
		\Xi^{hg}_{Y\cup X}		&\text{if }Y\neq\sen(\sn)\\
		\Delta^{hg}_Z			&\text{if }Y=\sen(\sn),
	\end{cases}$
	\item \label{DeltaOmega1} 	$\begin{cases}
		\Omega_{1,g}			& Y\neq\sen(\sn)\\
		\Phi^1_{S(g)}			& Y=\sen(\sn),
		\end{cases}$
	\item \label{DeltaOmega2} 	$\begin{cases}
		\Omega_{2,g}			& Y\neq\sen(\sn)\\
		\Omega_{3,g}			& Y=\sen(\sn),
		\end{cases}$
	\item \label{PhiPhi} 		$\begin{cases}
		\Phi^g_{TS}				& h\in\sn\setminus\an\\
		\Phi^1_{TS}				& h\in\an,
		\end{cases}$
	\item \label{PhiOmega1} 	$\begin{cases}
		\Omega_{1,g}			&\text{if }\rank(T)\neq 1\\
		\Gamma_{\psi_1}			&\text{if }\im(T)\subseteq E\\
		\Omega_{3,g}			&\text{if }T=Z,
		\end{cases}$
	\item \label{PhiOmega2} 	$\begin{cases}
		\Omega_{2,g}			&\text{if }\rank(T)\neq 1\\
		\Gamma_{\phi_g}			&\text{if }\im(T)\subseteq E\\
		\Omega_{3,g}			&\text{if }T=Z,
		\end{cases}$
	\item \label{Omega1Xi} 		$\begin{cases}
		\Omega_{1,(i\;j)^g}		&\text{if }\Fix(h)=\{i,j\}\\
		\Omega_{1,h^g}			&\text{otherwise,}
	\end{cases}$
	\item \label{Omega1Omega1} 	$\begin{cases}
		\Gamma_{\psi_1}			&\text{if }\phi_h\in E\\
		\Omega_{1,g}			&\text{if }\phi_h=\phi_1,
		\end{cases}$
	\item \label{Omega1Omega2} 	$\begin{cases}
		\Gamma_{\phi_g}			&\text{if }\phi_h\in E\\
		\Omega_{2,g}			&\text{if }\phi_h=\phi_1,
		\end{cases}$
	\item \label{Omega2Xi} 		$\begin{cases}
		\Omega_{2,(i\;j)^g}		&\text{if }\Fix(h)=\{i,j\}\\
		\Omega_{2,h^g}			&\text{otherwise,}
	\end{cases}$
	\item \label{Omega2Delta} 	$\begin{cases}
		\Omega_{2,h^g}			&\text{if }X\neq\sen(\sn)\\
		\Gamma_{\phi_1}			&\text{if }X=\sen(\sn),
		\end{cases}$
	\item \label{Omega2Phi}		$\begin{cases}
		\Gamma_{\phi_s}			&\text{if }S=S(s)\\
		\Omega_{2,hS}			&\text{if }\rank(S)\neq 1,
		\end{cases}$
	\item \label{Omega3Xi} 		$\begin{cases}
		\Omega_{3,(i\;j)^g}		&\text{if }\Fix(h)=\{i,j\}\\
		\Omega_{3,h^g}			&\text{otherwise,}
	\end{cases}$
	\item \label{Omega3Delta} 	$\begin{cases}
		\Omega_{3,h^g}			&\text{if }X\neq\sen(\sn)\\
		\Gamma_{\phi_1}			&\text{if }S=\sen(\sn),
		\end{cases}$
	\item \label{Omega3Phi}		$\begin{cases}
		\Omega_{3,hS}			&\text{if }\rank(S)\neq 1\\
		\Gamma_{\phi_s}			&\text{if }S=S(s),
	\end{cases}$
\end{enumerate}
\end{multicols}
\section{Automorphisms of \texorpdfstring{$\en_2(\sn)$}{} - the group of units of \texorpdfstring{$\en_3(\sn)$}{}}\label{sec:AutEnd2}
Now that we have determined the elements of $M:=\en_2(\sn)$, and laid out the multiplication table, we move on to finding the automorphisms of $M$. We saw in Theorem~\ref{tm:StructureOfAutEndSn} that the group of units of $M$ is the automorphism group $\au(\en(\sn))=\{\Psi_g\mid g\in\sn\}$ and is isomorphic to $\sn$. Thus, by Theorem~\ref{Tm:AutCompleteMonoids}, we have
	\[\au(\en_2(\sn))\cong\Ker(\mathfrak{R})\times\sn,\]
where $\Ker(\mathfrak{R})$ consists of all automorphisms of $\en_2(\sn)$ that restrict to the identity of $\au(\sn)$.

We will make use of the following convention: for $\Theta\in\{\Psi,\Xi,\Delta,\Phi,\Omega,\Gamma\}$ we will use $\Theta$ to denote the set of all elements of $\en_2(\sn)$ that are labelled by (some decoration of) $\Theta$. We will use the Anglo-Saxon rune $\os$ (pronounced \'{o}s) to denote elements of $\au(\en_2(\sn))$. Note that $\os$ induces an isomorphism of the group of units of $\en(\sn)$, that is, of $\{\Psi_g\mid g\in\sn\}$.
\begin{lemma}\label{lem:SetOfIdempotents}
The idempotents $E(\en_2(\sn))$ of $\en_2(\sn)$ may be listed as follows. We do not repeat the inherent conditions on $X,S$, etc.
	\begin{enumerate}
		\item $\Psi_1$;
		\item $\Xi^1_X$ for all $X$;
		\item $\Delta^1_X$ for all $X$;
		\item $\Phi^g_S$ such that $g\in(\sn\setminus\an)\cup\{1\}$ and $S^2=S$;
		\item $\Omega_{1,1}$;
		\item $\Omega_{2,g}$ for all $g$;
		\item $\Gamma_\lambda$ for all $\lambda$.
	\end{enumerate}
\end{lemma}
\begin{proof}
	We conduct the search for idempotents case by case depending on type.
	\begin{enumerate}
		\item Clearly $\Psi_1$ is the identity of $\en_2(\sn)$ and so the only idempotent in the group of units.
		\item If any $\Xi^g_X$ is an idempotent then $(\Xi^g_X)^2=\Xi^g_X$. Consulting Table~\ref{tab:table}, we have
			\[(\Xi^g_X)^2=\Xi^{g^2}_X.\]
			So $\Xi^g_X$ is idempotent if and only if $g^2=g$, that is, if and only if $g$ is the identity $1$ of $\sn$.
		\item A similar argument gives that $\Delta^g_X$ is idempotent if and only if $g=1$.
		\item Assume that $\Phi^g_S$ is an idempotent. Consulting Table~\ref{tab:table} we see that $g$ must either be in $\sn\setminus\an$ or be the identity of $\sn$. If $g$ satisfies these conditions then $(\Phi^g_S)^2=\Phi^g_{S^2}$. We deduce that $\Phi^g_S$ is idempototent if and only if the stated conditions hold.
		\item From Table~\ref{tab:table} it is clear that the only $g$ for which $\Omega_{1,g}$ is idempotent is $g=1$.
		\item From Table~\ref{tab:table} it is clear that $(\Omega_{2,g})^2=\Omega_{2,g}$ for all $g$.
		\item Clearly all $\Gamma_{\lambda}$ are idempotent as they are precisely the constant endomorphisms.
	\end{enumerate}
	
		We remark that from Table~\ref{tab:table} it is clear that $(\Omega_{3,g})^2=\Gamma_{\phi_1}$ for all $g$, so that no $\Omega_{3,g}$ is idempotent. This ends the discussion of the possible types of elements.
\end{proof}
In order to fully describe the automorphisms of $\en_2(\sn)$, we will show that some well-chosen sets are characteristic, and that some specific idempotents are fixed by all automorphisms. This will allow us to gradually show that all endomorphisms are fixed by elements of $\Ker(\mathfrak{R})$. We recall that for an element $m$ of monoid $M$ with group of units $G$, we denote by $C(m)$ the set of {\em elements of $G$} that commute with $m$. If $m$ is the identity $e$ of $M$, then it follows that $C(e)=G$.
\begin{lemma}\label{lem:SameCentIdent}
	Let $\Lambda \in \en_2(\sn)$. Then
    \[\mathcal{C}:=\{\Lambda\mid C(\Lambda)=\au(\en(\sn))\}=\{\Lambda\mid C(\Lambda)=C(\Psi_1)\}\]
    is a characteristic subset.
\end{lemma}
\begin{proof}
We have already remarked that any $\os\in\au(\en_2(\sn))$ induces an automorphism of $\{\Psi_g\mid g\in\sn\}$ and so in particular is a bijection. Hence
	\[\begin{array}{rcll}
		C(\Lambda)=\au(\en(\sn))	&\Leftrightarrow& \Lambda\Psi_g=\Psi_g\Lambda						&\text{ for all }\Psi_g\\
									&\Leftrightarrow& (\Lambda\Psi_g)\os=(\Psi_g\Lambda)\os				&\text{ for all }\Psi_g\\
									&\Leftrightarrow& (\Lambda\os)(\Psi_g\os)=(\Psi_g\os)(\Lambda\os)	&\text{ for all }\Psi_g\\
									&\Leftrightarrow& (\Lambda\os)\Psi_h=\Psi_h(\Lambda\os)				&\text{ for all }\Psi_h\\
									&\Leftrightarrow& C(\Lambda\os)=\au(\en(\sn)).
	\end{array}\]
The result follows.
\end{proof}
\begin{lemma}\label{lem:int}\label{CharIdempts}
	The set $E(\en_2(\sn))\cap\mathcal{C}$ is characteristic. Further, the elements of $E(\en_2(\sn))\cap\mathcal{C}$ may be listed as follows (with the same enumeration as in Lemma~\ref{lem:SetOfIdempotents}):
	\begin{enumerate}
		\item $\Psi_1$;
		\item $\Xi^1_X$ for any $X$;
		\item $\Delta^1_X$ for any $X$;
		\item $\Phi^1_Z$;
		\item $\Omega_{1,1}$;
		\item[(7)] $\Gamma_\lambda$ where $\lambda$ is either $\psi_1$ or $\phi_1$.
	\end{enumerate}
\end{lemma}
\begin{proof}
The first statement follows from Example~\ref{ex:ch}. It remains to find the idempotents in $\mathcal{C}$. We search our list of idempotents by type. The element $\Psi_1$ lies in $\mathcal{C}$ since it is the monoid identity, so that (1) holds. It is immediate from the multiplication table that (2), (3) and (5) hold. For (4), consider an idempotent  $\Phi^g_S$, so that $g\in \sn\setminus\an$ or $g=1$, and $S=S^2$.  We have that $\Phi^g_S\Psi_t=\Psi_t\Phi^g_S=\Phi^g_S$ if and only if $g^t=g$ and $S\Psi_t=S$ for all $\Psi_t$. The latter condition says that  for all $\phi_k$ lying in $\im(S)$, we require $\phi_k\Psi_t=\phi_k$. We deduce that $\Phi^g_S\in\mathcal{C}$ if and only if $g^t=g$ and $k^t=k$ for all $t\in\sn$, that is, if and only if $g$ and $k$ lie in the centre of $\sn$. Now (4) follows since $\sn$ has trivial centre.

We now show that no $\Omega_{2,g}$ lies in $\mathcal{C}$. Suppose to the contrary that $\Omega_{2,g}\in\mathcal{C}$, that is, $\Omega_{2,g}$ commutes with every $\Psi_t$. It follows from Table~\ref{tab:table} that $\Omega_{2,g}=\Omega_{2,g^t}$ for all $t\in\en(\sn)$ and so $g^t=g$ for all $t\in\en(\sn)$. But here $g$ must be odd, and so there are no possibilities for $g$. Similarly, if $\Gamma_\lambda$ commutes with every $\Psi_t$, then from Table~\ref{tab:table} it is clear we require that $\lambda=\lambda\Psi_t$ for all $t\in\en(\sn)$. This is true if $\lambda=\psi_1$. Otherwise, if $\lambda=\phi_g$, then we would require $g=g^t$ for all $t$ and hence $g=1$ as required.
\end{proof}
\begin{lemma}\label{IdempotentHasse}
	Any $\os\in\au(\en_2(\sn))$ fixes $\Omega_{1,1}$, $\Delta^1_Z$ and $\Phi^1_Z$.
\end{lemma}
\begin{proof}
	All three elements lie in the characteristic subset described in Lemma \ref{CharIdempts}. Now let $\Lambda$ be any element of $\{\Psi_1,\Xi^1_X,\Delta^1_X\mid X\neq\sen(\sn)\}$. The down set of $\Lambda$ (with respect to the natural partial order on idempotents) in the set $E(\en_2(\sn))\cap\mathcal{C}$ is given in Figure~\ref{fig:down set}, as may be verified using Table~\ref{tab:table}.
	\begin{figure}[H]
	\begin{center}
		\begin{tikzcd}
                                      		 		& \Lambda \arrow[ld] \arrow[rd] &                                \\
			{\Omega_{1,1}} \arrow[rrdd] \arrow[dd] 	&                               & \Delta^1_Z \arrow[d]           \\
                                       				&                               & \Phi^1_Z \arrow[d] \arrow[lld] \\
			\Gamma_{\psi_1}                        	&                               & \Gamma_{\phi_1}
		\end{tikzcd}
	\end{center}\caption{The down set of $\Lambda$ in $E(\en_2(\sn))\cap\mathcal{C}$}\label{fig:down set}\end{figure}
	So, any such $\Lambda$ has a down set in $E(\en_2(\sn))\cap\mathcal{C}$ of size at least 6. Comparing this to the other elements of $E(\en_2(\sn))\cap\mathcal{C}$, we have
	\[\setlength{\extrarowheight}{2mm}
    \begin{array}{l|*{1}{l}}
	\text{Idempotent}	& \text{Size of down set}\\
	\hline
	\Lambda				& \geq 6\\
	\Omega_{1,1}		& 3\\
	\Delta^1_Z 			& 4\\
	\Phi^1_Z			& 3\\
	\Gamma_{\phi_1}		& 1\\
	\Gamma_{\psi_1}		& 1
    \end{array}\]

By Lemma~\ref{lem:imageidentityzero}, any automorphism must preserve the natural order of idempotents. Thus, the down set of any such idempotent $\Lambda$ is the same size as the down set of any $\Lambda\os$ that lies in $E(\en_2(\sn))\cap\mathcal{C}$. Thus, $\Delta^1_Z$ must be fixed by all $\os$ in $\au(\en_2(\sn))$. It follows that $\Phi^1_Z\os$ must have a down set of size 3 and be below $\Delta^1_Z$. Thus, $\Phi^1_Z\os=\Phi^1_Z$ for all $\os$ in $\au(\en_2(\sn))$. Finally, as the down set of $\Omega_{1,1}$ is 3 and the only other element with equal sized down set is $\Phi^1_Z$ (which is fixed), we must have that $\Omega_{1,1}\os=\Omega_{1,1}$ for all $\os$ in $\au(\en_2(\sn))$.
\end{proof}
We now focus on showing that any $\os$ in $\Ker(\mathfrak{R})$, that is, any automorphism of $\en_2(\sn)$ that restricts to the identity on $\au_2(\sn)$, is the identity map.
\begin{lemma}\label{lem:Omega1Fixed}
	Let $\os$ be in $\Ker(\mathfrak{R})$. Then $\os$ fixes all elements $\Omega_{1,g}$.
\end{lemma}
\begin{proof}
	We saw in Lemma~\ref{IdempotentHasse} that $\Omega_{1,1}$ must be fixed by any $\os$ in $\au(\en_2(\sn))$. From Table~\ref{tab:table} we have that $\Omega_{1,1}$ is a left identity for any $\Omega_{1,g}$. Thus, we must have $\Omega_{1,1}\os\Omega_{1,g}\os=\Omega_{1,g}\os$. That is, $\Omega_{1,1}$ must also be a left identity for $\Omega_{1,g}\os$.
	
	As $\Omega_{1,1}$ is rank 2, it cannot be the left identity of any element of rank greater than 2. Since $\Gamma$ is characteristic by Remark~\ref{rem:Gamma_min_ideal}, it follows that $\Omega_{1,g}\os$ must also have rank 2. The only rank 2 elements are maps of the form $\Phi^1_S$ for $\rank(S)=1$, $\Omega_{1,h}$, $\Omega_{2,h}$, or $\Omega_{3,h}$.
	
	The elements $\Phi^1_S$ and $\Omega_{2,h}$ are idempotents while those of the form $\Omega_{1,g}$ are not whenever $g\neq 1$. As automorphisms take only idempotents to idempotents, the only possibilities remaining are $\Omega_{1,g}\os=\Omega_{1,h}$ or $\Omega_{1,g}\os=\Omega_{3,h}$ for some $h$. But $\Omega_{1,1}$ is not a left identity for any $\Omega_{3,h}$. This leaves us with the only possibility being $\Omega_{1,g}\os=\Omega_{1,h}$ for some $h$. Using Lemma~\ref{lem:SameCentUnderAut} in the situation where $M=\en_2(\sn)$ gives us that $C(\Omega_{1,g})=C(\Omega_{1,h})$, which in turn implies that $C(g)=C(h)$. From Corollary~\ref{cor:centralisers} we deduce that $g=h$.
\end{proof}
\begin{lemma}\label{lem:Omega2Fixed}
	Let $\os$ be in $\Ker(\mathfrak{R})$. Then $\os$ fixes all elements $\Omega_{2,g}$.
\end{lemma}
\begin{proof}
	All $\Omega_{2,g}$ elements are idempotents and from Table~\ref{tab:table} we have $\Omega_{1,1}\Omega_{2,g}=\Omega_{2,g}$. Thus, as $\os$ fixes $\Omega_{1,1}$ and sends idempotents to idempotents, $\Omega_{2,g}\os$ must be some idempotent such that $\Omega_{1,1}(\Omega_{2,g}\os)=\Omega_{2,g}\os$. As in Lemma~\ref{lem:Omega1Fixed} we deduce that $\Omega_{2,g}\os$ has rank 2. Looking at our multiplication table, the only rank 2 elements that have $\Omega_{1,1}$ as a left identity are those of the form $\Omega_{1,h}$ or $\Omega_{2,h}$. From Lemma~\ref{lem:Omega1Fixed} (applied to $\os^{-1}$) we deduce that $\Omega_{2,g}\os$ cannot equal any $\Omega_{1,h}$. Thus $\Omega_{2,g}\os=\Omega_{2,h}$ for some $h$. Calling upon the general theory of Lemma~\ref{lem:SameCentUnderAut} with $M=\en_2(\sn)$ gives us that $C(\Omega_{2,g})=C(\Omega_{2,h})$, which in turn implies that $C(g)=C(h)$. From Corollary~\ref{cor:centralisers} we deduce that $g=h$.
\end{proof}
\begin{lemma}\label{lem:Omega3Fixed}
Let $\os$ be in $\Ker(\mathfrak{R})$. Then $\os$ fixes all elements $\Omega_{3,g}$.
\end{lemma}
\begin{proof}
	For any $\Omega_{3,g}$ we have that $\Omega_{3,g}^{2}=\Gamma_{\phi_1}$. Thus, as $\Gamma$ is characteristic, $(\Omega_{3,g}\os)^2=\Gamma_\lambda$ for some $\lambda$. Looking at our composition table, the only elements for which the square lies in $\Gamma$ are those of the form $\Omega_{1,h}$ or $\Omega_{3,h}$. As the elements $\Omega_{1,h}$ are fixed by $\os$, we must have that $\Omega_{3,g}\os=\Omega_{3,h}$ for some $h$.

Similarly to Lemma~\ref{lem:Omega1Fixed} we obtain that $C(g)=C(h)$ and then $g=h$. Thus $\Omega_{3,g}\os=\Omega_{3,g}$, as required.
\end{proof}
\begin{lemma}\label{lem:GammaFixed}
	Let $\os$ be in $\Ker(\mathfrak{R})$. Then $\os$ fixes all elements in the minimal ideal $\Gamma$.
\end{lemma}
\begin{proof}
By Remark~\ref{rem:Gamma_min_ideal}, we know that $\Gamma$ is characteristic. So consider $\Gamma_\lambda$ and set $\Gamma_\lambda\os=\Gamma_{\lambda'}$. Hence, by Lemma~\ref{lem:SameCentUnderAut} we have that $C(\Gamma_{\lambda})=C(\Gamma_{\lambda'})$, and it follows by a now standard argument that $C(\lambda)=C(\lambda')$.
	
	Recall that $\lambda$ is one of: $\psi_1$, $\phi_1$, or $\phi_g$ for some odd $g$ of order 2 such that $g\neq 1$. Notice that $C(\psi_1)=\{\psi_t\mid t\in\sn\}$ and that $C(\phi_{g})=\{\psi_t\mid t\in C(k)\}$. Since $\sn$ is centreless, it follows that $\os$ fixes both $\Gamma_{\phi_1}$ and $\Gamma_{\psi_1}$, or else transposes them. We use Lemma~\ref{lem:Omega3Fixed} to show the former must hold.

For any $\Omega_{3,g}$ we have
	\[\Gamma_{\phi_1}\os=(\Omega_{3,g}^2)\os=(\Omega_{3,g}\os)^2=\Omega_{3,g}^2=\Gamma_{\phi_1}.\]
So $\Gamma_{\phi_1}$ and, consequently, $\Gamma_{\psi_1}$ are both fixed by $\os$.
	
	The remaining option is that $\lambda=\phi_g$ and $\lambda'=\phi_h$ for some odd $g,h$ of order 2. But then we obtain $C(g)=C(h)$ and $g=h$ as before.

	Thus, all $\Gamma_\lambda$ are fixed by $\os$.
\end{proof}
\begin{lemma}\label{lem:DeltaXiInvatiant}
    For any $\os$ in $\Ker(\mathfrak{R})$ we have that $\Xi\os=\Xi$ and $\Delta\os=\Delta$. Further,
 	\[\Xi^1_X\os=\Xi^1_Y\Leftrightarrow\Xi^g_X\os=\Xi^g_Y\text{ and }\Delta^1_X\os=\Delta^1_Y\Leftrightarrow\Delta^g_X\os=\Delta^g_Y\]
for all $g$ in $\sn$.
\end{lemma}
\begin{proof}
	The elements $\Xi^g_X$ and $\Delta^g_X$ are the products $\Xi^1_X\Psi_g$ and $\Delta^1_X\Psi_g$ respectively. As $\os$ is in $\Ker(\mathfrak{R})$, it must fix $\Psi_g$. Thus,
	\[\Xi^g_X\os=(\Xi^1_X\Psi_g)\os=\Xi^1_X\os\Psi_g\os=\Xi^1_X\os\Psi_g,\]
and similarly,
	\[\Delta^g_X\os=\Delta^1_X\os\Psi_g\os=\Delta^1_X\os\Psi_g.\]
Thus the second statement follows, and in order to determine $\Xi^g_X\os$ or $\Delta^g_X\os$, we only need determine $\Xi^1_X\os$ and $\Delta^1_X\os$.
	
From Lemma~\ref{IdempotentHasse} we know that $\Delta^1_Z\os=\Delta^1_Z$ so we take $X\neq Z$. In the proof of Lemma~\ref{IdempotentHasse} we remarked that $\Delta^1_X$ and $\Xi^1_X$ lie in $E(\en_2(\sn))\cap \mathcal{C}$ and they lie at the point $\Lambda$ on our Hasse diagram from Figure~\ref{fig:down set}, and they are precisely the elements that have a down set in $E(\en_2(\sn))\cap \mathcal{C}$ of size at least 6. Since $E(\en_2(\sn))\cap\mathcal{C}$ is characteristic, $\os$ must permute the set $\{\Delta^1_X,\Xi^1_X\mid X\neq\sen(\sn)\}$, and so it follows from the first paragraph that $\Delta\cup\Xi$ is characteristic in $\en_2(\sn)$.
	
Suppose that $\Delta^1_X\os=\Xi^1_Y$. Note that in this case, $n$ is a multiple of 4 and so we can pick a $g\in\sn\setminus\an$ such that $\Fix(g)=\{i,j\}$. Then, using Lemma~\ref{lem:Omega1Fixed}, we have
	\[(\Omega_{1,g}\Delta^1_X)\os=\Omega_{1,g}\os=\Omega_{1,g}.\]
However,
	\[\Omega_{1,g}\os\Delta^1_X\os=\Omega_{1,g}\Xi^1_Y=\Omega_{1,(i\;j)},\]
a contradiction. Thus, no $\Delta^1_X$ is sent to any $\Xi^1_Y$. It follows that each $\Delta^1_X$ is sent to some $\Delta^1_Y$. Thus, $\Delta\os=\Delta$ and, as $\Delta\cup\Xi$ is characteristic, $\Xi\os=\Xi$, as required.
\end{proof}
\begin{lemma}\label{lem:PhisCharUnderOs}
	For any $\os$ in $\Ker(\mathfrak{R})$ we have that $\Phi\os=\Phi$.
\end{lemma}
\begin{proof}
	Our monoid $\en_2(\sn)$ is the union of $\Psi\cup\Xi\cup\Delta\cup\Phi\cup\Omega\cup\Gamma$. By definition $\os$ fixes $\Psi$, and we have shown in Lemmas~\ref{lem:Omega1Fixed} to~\ref{lem:Omega3Fixed} that $\os$ fixes $\Omega$, in Lemma~\ref{lem:GammaFixed} that $\os$ fixes $\Gamma$ and then in Lemma~\ref{lem:DeltaXiInvatiant} that $\os$ fixes $\Delta$ and $\Xi$. Consequently, $\os$ fixes the remaining set of elements, namely $\Phi$.
\end{proof}
\begin{lemma}\label{lem:PhiTopIndexFixed}
Let $\os$ be in $\Ker(\mathfrak{R})$. Then $\os$ fixes all elements in the set
	\[\{\Phi^g_Z\mid g^2=1\}.\]
\end{lemma}
\begin{proof}
	Consider any $\Phi^g_Z$. We saw in Lemma~\ref{lem:PhisCharUnderOs} that $\Phi\os=\Phi$, so $\Phi^g_Z\os=\Phi^h_S$ for some $h$ and some $S$. We also saw in Lemma~\ref{IdempotentHasse} that $\Delta^1_Z$ is fixed by $\os$.
	
As $\Phi^g_Z\Delta^1_Z=\Phi^g_Z$ we must have $\Phi^g_Z\os\Delta^1_Z\os=\Phi^g_Z\os$ and so $\Phi^h_S\Delta^1_Z=\Phi^h_S$. This implies that $S=Z$. So $\Phi^g_Z\os=\Phi^h_Z$ for some $h$.
	
By Lemma~\ref{lem:SetOfIdempotents}, $\Phi^g_Z$ is an idempotent if and only if $g\in(\sn\setminus\an)\cup\{1\}$, and from Lemma~\ref{IdempotentHasse} we know that $\Phi_Z^1$ is fixed by $\os$. Since $\os$ must send idempotents to idempotents, it follows that $g$ and $h$ are either both in $\sn\setminus\an$, or both in $\an\setminus\{1\}$, or $g=h=1$.
	
Using Lemma~\ref{lem:SameCentUnderAut} in the current situation, we get that $C(\Phi^g_Z)=C(\Phi^h_Z)$ which leads to $C(g)=C(h)$ and from Corollary~\ref{cor:centralisers} that $g=h$. Thus $\Phi^g_Z\os=\Phi^g_Z$ for all $g$, as required.
\end{proof}
\begin{lemma}\label{PhiTopIndex}
	Let $\os$ be in $\Ker(\mathfrak{R})$. Then for any $\Phi^g_S$ we have $\Phi^g_S\os=\Phi^g_T$ for some $T$. That is, the top index of an element in $\Phi$ stays the same.
\end{lemma}
\begin{proof}
	Recall from Lemma~\ref{lem:PhiTopIndexFixed} that all $\Phi^g_Z$ are fixed and from Lemma~\ref{IdempotentHasse} that $\Delta^1_Z$ is fixed. Let $\Phi^g_S$ be in $\en_2(\sn)$. Then $\Phi^g_S\Delta^1_Z=\Phi^g_Z$. From Lemma~\ref{lem:PhisCharUnderOs} $\Phi^g_S\os=\Phi^h_T$, for some $\Phi^h_T$. Then
	\[(\Phi^g_S\Delta^1_Z)\os=\Phi^g_Z\os=\Phi^g_Z\]
and
	\[\Phi^g_S\os\Delta^1_Z\os=\Phi^h_T\Delta^1_Z=\Phi^h_Z.\]
So $\Phi^g_Z=\Phi^h_Z$ which implies that $g=h$. That is, $\Phi^g_S\os$ must be $\Phi^g_T$ for some $T$.
\end{proof}
\begin{lemma}\label{PhiRank1}
	Let $\os$ be in $\Ker(\mathfrak{R})$. Then $\os$ fixes all $\Phi^g_S$ where $\rank(S)=1$.
\end{lemma}
\begin{proof}
	Recall from Lemma~\ref{lem:PhiTopIndexFixed} that $\Phi^g_Z$ is fixed for all $g$, so we only need focus on those $\Phi^g_S$ such that $\rank(S)=1$ but $S\neq Z$. So consider such $\Phi^g_S$, that is, $S=S(h)$ for some $\phi_h\in E$. From Lemma~\ref{PhiTopIndex}, we have that $\Phi^g_S\os=\Phi^g_T$ for some $T\neq Z$.
	
For all $\Omega_{1,t}$ in $\en_2(\sn)$, we have
	\[\Omega_{1,1}\Phi^g_S=\Omega_{1,t}\Phi^g_S=\Omega_{1,h}.\]
Recall from Lemma~\ref{lem:Omega1Fixed} that all $\Omega_{1,t}$ are fixed by $\os$. Thus, applying $\os$ yields the following for all $\Omega_{1,t}$ in $\en_2(\sn)$.
	\[\Omega_{1,1}\Phi^g_T=\Omega_{1,t}\Phi^g_T=\Omega_{1,h}\]
so that $\phi_1T=\phi_tT$ for all $\phi_t\in E$. Looking at the conditions on $T$, we must have that $\im(T)=\{\phi_h\}$. We deduce that $S=T$ and $\Phi^g_S=\Phi^g_T$ as required.
\end{proof}
\begin{corollary}\label{cor:phis}
	Let $\os$ be in $\Ker(\mathfrak{R})$. Let $\mathcal{T}$ denote the set of all $\Phi^1_S$ and let $\mathcal{S}$ denote the set of all $\Phi^1_S$ where $rank(S)\neq 1$. Then $\mathcal{S}$ is a subsemigroup of $\mathcal{T}$. Further, $\mathcal{T}\os=\mathcal{T}$ and $\mathcal{S}\os=\mathcal{S}$.
\end{corollary}
\begin{proof}
	It is easy to check that $\mathcal{T}$ is a semigroup with subsemigroup $\mathcal{S}$. That $\mathcal{T}\os=\mathcal{T}$ follows from Lemma~\ref{PhiTopIndex}. Since we know from Lemma~\ref{PhiRank1} that $\Phi^1_S\os=\Phi^1_S$ for all $S$ with $\rank(S)=1$, the second statement follows.
\end{proof}
\begin{lemma}\label{ResToE}
	Let $\os$ be in $\Ker(\mathfrak{R})$ and let $\Lambda$ be in $\Delta\cup\Xi\cup\Phi$. Then for any $\phi_g\in E$,
	\[\phi_g\Lambda=\phi_h\Leftrightarrow\phi_g(\Lambda\os)=\phi_h.\]
That is, $\Lambda\mid_E=(\Lambda\os)\mid_E$.
\end{lemma}
\begin{proof}
	Any $\Lambda$ in these sets is such that $E\Lambda\subseteq E\cup\{\phi_1\}$. Thus, for any $\phi_g\in E$,
		\[\phi_g\Lambda=\phi_h\Leftrightarrow\Omega_{1,g}\Lambda=\Omega_{1,h}.\]
	From Lemma~\ref{lem:Omega1Fixed}, we have that $\os$ fixes $\Omega_{1,g}$ and $\Omega_{1,h}$. So if $\Omega_{1,g}\Lambda=\Omega_{1,h}$ then applying $\os$ yields
		\[\Omega_{1,g}(\Lambda\os)=\Omega_{1,h}.\]
    Since we also know that $\Lambda\os\in\Delta\cup\Xi\cup\Phi$ we deduce that this is equivalent to $\phi_g(\Lambda\os)=\phi_h$, as required.
\end{proof}
Notice that Lemma~\ref{ResToE} comes from a representation of how $\Lambda$ acts on $E$ using particular products in $\en_2(\sn)$. This gave us information on how $\Lambda\os$ acts on $E$. In order to invoke a similar trick for how $\Lambda\os$ acts on $A$ we will utilise a particularly `simple' set of elements in $\Phi$. These will be idempotent elements $\Phi^1_S$ where $S$ has a very restricted form (mapping all of $E$ to a single element and mapping all of $A$ to a single element). To simplify our notation we define them as follows:
\[\Phi_{g,k}:\omega\mapsto\begin{cases}
	\psi_1	&\text{if }\omega\in\au(\sn),\\
	\phi_g	&\text{if }\omega\in E,\\
	\phi_k	&\text{if }\omega\in A,\\
	\phi_1	&\text{if }\omega=\phi_1,
\end{cases}\]
where $\phi_g$ is in $E$ and $\phi_k$ is in $A\cup\{\phi_1\}$. Note that every such element has image with 4 elements (if $k\neq 1$) and 3 elements (if $k=1$).

By abuse of notation we will write the product $\Phi_{g,k}\Phi^t_S$, where $S$ is not constant, as $\Phi_{gS,kS}$ where $\Phi_{gS,kS}=\Phi_{g',k'}$ if and only if $\phi_gS=\phi_{g'}$ and $\phi_kS=\phi_{k'}$.
\begin{lemma}\label{lem:PhigkOsIsPhigk'}
	For all $\os$ in $\Ker(\mathfrak{R})$ we have that $\Phi_{g,k}\os=\Phi_{g,k'}$ for some $\phi_{k'}\in A$. Further, for all $\phi_g\in E$, we have $\Phi_{g,1}\os=\Phi_{g,1}$.
\end{lemma}
\begin{proof}
	The element $\Phi_{g,k}$ is $\Phi^1_S$ for some $S$ and from Lemma \ref{PhiTopIndex} we know that $\Phi^1_S\os=\Phi^1_T$ for some $T$. Further, from Lemma \ref{PhiRank1}, we know that as $\rank(S)$ is not 1, neither is $\rank(T)$.  This can only happen when $ET\subseteq E$, $AT\subseteq A\cup\{\phi_1\}$ and $\phi_1T=\phi_1$. Using Lemma \ref{ResToE}, we also know that $\Phi_{g,k}$ and $\Phi^1_T$ must restrict to the same map on $E$. That is, $ET=\{\phi_g\}$. It remains only to show that $AT=\{\phi_{k'}\}$ for some $\phi_{k'}$.
	
	We start by letting $k=1$ and show that $\Phi_{g,1}$ is fixed by $\os$ for all $g$. First notice that $\mathcal{S}$ contains all the elements $\Phi_{h,s}$. Now observe that for any $\Phi^1_A,\Phi^1_B\in\mathcal{S}$ we have that $\Phi^1_A\Phi_{g,1}\Phi^1_B=\Phi_{gB,1}$. Thus, the set $\{\Phi_{g,k}\mid\phi_g\in E,k=1\}$ is the minimal ideal of $\mathcal{S}$. From Corollary~\ref{cor:phis}, we have that $\mathcal{S}\os=\mathcal{S}$, and it follows from Example~\ref{ex:ch} that $\{\Phi_{g,k}\mid\phi_g\in E,k=1\}\os=\{\Phi_{g,k}\mid\phi_g\in E,k=1\}$. Now, if $\Phi_{g,1}\os=\Phi_{h,1}$ then we use Lemma~\ref{lem:SameCentUnderAut} and notice that $C(\Phi_{g,1})=C(\Phi_{h,1})$ implies that $C(g)=C(h)$, which from Corollary~\ref{cor:centralisers} yields that $g=h$. Thus, $\Phi_{g,1}$ is fixed by $\os$.
	
	Suppose now that $k$ is not 1: we will show that $AT\subseteq A$. If there is some $\phi_{s}$ in $A$ such that $\phi_{s}T=\phi_1$ then
		\[\Phi_{g,s}\Phi^1_T=\Phi_{g,1}\]
	and applying $\os^{-1}$ (recalling that $\os^{-1}\in\Ker(\mathfrak{R})$) yields
		\[\Phi_{g,s}\os^{-1}\Phi_{g,k}=\Phi_{g,1}.\]

	Thus, $\Phi_{g,s}\os^{-1}$ must be equal to $\Phi_{g,1}$, implying that $s=1$, a contradiction. Thus, $AT\subseteq A$.
	
	Now we only need to show that $\rank(T)=3$, which forces $T$ to restrict to a constant map of $A$. We will do this by using the down set of $\Phi_{g,k}$ in the natural partial order of idempotents in $\Phi$. We will restrict our search to $\mathcal{S}=\mathcal{S}\os$ and denote the down set of $\Phi^1_B$, where $B=B^2$, in the idempotents of $\mathcal{S}$ by \[\Phi^{1\hspace{1mm}\downarrow}_B=\{\Phi^1_A\mid \Phi^1_A\Phi^1_B=\Phi^1_B\Phi^1_A=\Phi^1_A\}.\]
	
Recall that $\Phi_{g,k}\os=\Phi^1_T$ where $T$ is not constant and we must have that $T$ is idempotent. As $\os$ is an automorphism, the sizes of $\Phi_{g,k}^{\hspace{3mm}\downarrow}$ and $\Phi^{1\hspace{1mm}\downarrow}_T$ must be the same. Let $\Phi^1_B$ be strictly below $\Phi_{g,k}$. It is easy to see that $\rank(B)$ must be strictly less than 3, so that as $\rank(B)\neq 1$ we have $\rank(B)=2$. So $\Phi^1_B$ has the form $\Phi_{h,1}$ for some $h$. Further, as $\Phi_{h,1}\Phi_{g,k}=\Phi_{h,1}$, we must have $h=g$. That is, the down set of any $\Phi_{g,k}$ in the natural order of idempotents is exactly $\{\Phi_{g,k},\Phi_{g,1}\}$. Thus, if $\Phi_{g,k}\os=\Phi^1_T$ then the down set of $\Phi^1_T$ must be $\{\Phi^1_T,\Phi_{g,1}\}$. If $\rank(T)$ is strictly greater than 3, then there exists $\phi_u$ and $\phi_v$ in $A$ that are fixed by $T$ (since $T$ is idempotent) and then
	\[\Phi_{g,u}\Phi^1_T=\Phi_{g,u}=\Phi^1_T\Phi_{g,u}\quad\text{ and }\Phi_{g,v}\Phi^1_T=\Phi_{g,v}=\Phi^1_T\Phi_{g,v}.\]
So $\Phi_{g,u}$ and $\Phi_{g,v}$ are in the down set of $\Phi^1_T$ and since $\Phi_{g,1}$ must also be in the down set of $\Phi^1_T$, the size of this down set is strictly greater than 3, a contradiction. We deduce that $\rank(T)=3$. This leaves us with the only possibility, that $AT=\{\phi_{k'}\}$ for some $\phi_{k'}$ in $A$, as required.
\end{proof}
\begin{lemma}\label{lem:Phigh-Char}
	Let $\phi_k\in A$. Then there is some $\phi_g\in E$ such that $\Phi_{g,k}\os=\Phi_{g,k}$ for all $\os$ in $\Ker(\mathfrak{R})$.
\end{lemma}
\begin{proof}
	Since $k\in\an$ and is of order 2 we may invoke Lemma~\ref{lem:each_g_has_its_nice_k} to choose an odd $g\in\sn$ of order 2 such that whenever $C(g)\cap C(k)=C(g)\cap C(k')$ with $k'\in\an$ of order two, we must have $k'=k$.
Moreover, $g$ can be taken to be any transposition that is a factor of $k$.

We have that $\phi_g\in E$ and by Lemma~\ref{lem:PhigkOsIsPhigk'}, we know that for $\os$ in $\Ker(\mathfrak{R})$ we have $\Phi_{g,k}\os=\Phi_{g,k'}$ for some $\phi_{k'}\in A$. Additionally, from Lemma~\ref{lem:SameCentUnderAut} with $M=\en_2(\sn)$, we must have that $C(\Phi_{g,k})=C(\Phi_{g,k'})$. Let $T:\sen(\sn)\rightarrow\{\phi_g,\phi_k,\phi_1\}$ and let $S:\sen(\sn)\rightarrow\{\phi_g,\phi_{k'},\phi_1\}$ be such that $\Phi_{g,k}=\Phi^1_T$ and $\Phi_{g,k'}=\Phi^1_S$. Then using Table~\ref{tab:table} we see that $\Psi_h\in C(\Phi_{g,k})$ if and only if $\Phi_{g,k}=\Phi_{g^h,k^h}$, that is, $h\in C(g)\cap C(k)$. Thus $C(\Phi^1_T)=C(\Phi^1_S)$ leads to $C(g)\cap C(k)=C(g)\cap C(k')$. Consequently, $k'=k$, which shows that the element $\Phi_{g,k}$ is fixed by $\os$.
\end{proof}
\begin{corollary}\label{lem:DeltasFixed}
	Let $\os$ in $\Ker(\mathfrak{R})$. Then $\os$ fixes all $\Delta^g_X$.
\end{corollary}
\begin{proof}
	From Lemma~\ref{lem:DeltaXiInvatiant} it is enough to show that $\os$ fixes any $\Delta^1_X$. From Lemma~\ref{lem:SetOfIdempotents} the elements $\Delta^1_X$ are precisely the elements of $\Delta$ that are idempotent and certainly $\os$ maps idempotents to idempotents. We know that $\Delta\os=\Delta$ and consequently for any $X$ we have $\Delta^1_X\os=\Delta^1_{X'}$ for some $X'$. We know that $\os$ fixes $\Delta^1_Z$ so we may assume that $X\neq\sen(\sn)$. It follows that $X$ is a union of conjugacy classes of $A\cup\{\phi_1\}$, containing $\phi_1$. From Table~\ref{tab:table} we see that if $X=X_1\cup X_2\cup\cdots\cup X_m$, where all $X_i$ are conjugacy classes in $A$, then $\Delta^1_X=\Delta^1_{X_1}\Delta^1_{X_2}\cdots\Delta^1_{X_m}$. We now observe that the set of elements of the form $\Delta^1_X$ is a finite semilattice, isomorphic to the set of conjugacy classes of $A$ under union. Since $\os$ is a homomorphism, for any conjugacy class $Y$ of $A$ we have that $\Delta^1_Y\os=\Delta^1_{Y'}$ for some conjugacy class $Y'$ of $A$ and it is enough to show $\os$ fixes each such $\Delta^1_Y$.

  To this end let $Y=\phi_k\au(\sn)$ for some $\phi_k\in A$. By Lemma~\ref{lem:Phigh-Char} there exists $\phi_g\in E$ such that $\Phi_{g,k}\os=\Phi_{g,k}$ and a calculation gives $\Phi_{g,k}\Delta^1_Y=\Phi_{g,1}$. Also from Lemma~\ref{lem:PhigkOsIsPhigk'}, we know that $\Phi_{g,1}$ is fixed by $\os$. Thus,
	\[\Phi_{g,1}=\Phi_{g,1}\os=(\Phi_{g,k}\Delta^1_Y)\os=\Phi_{g,k}\Delta^1_{Y'}.\]
That is, $\phi_k$ must be in $Y'$ and so $Y'=\phi_k\au(\sn)=Y$. This completes the proof of the corollary.
\end{proof}
\begin{corollary}
	Let $\os$ in $\Ker(\mathfrak{R})$. Then $\os$ fixes all $\Xi^g_X$.
\end{corollary}
\begin{proof}
	Arguing just as in Corollary~\ref{lem:DeltasFixed} one finds that the elements $\Xi_X^1$ form a finite semilattice from which it follows that if $Y=\phi_k\au(\sn)$, then $\Xi^1_Y\os=\Xi^1_{Y'}$ where $Y'=\phi_{k'}\au(\sn)$ for some $\phi_{k'}$. It then again suffices to show that $\os$ fixes each $\Xi^1_Y$.
	
	To this end let $Y$ and $Y'$ be as above. By Lemma~\ref{lem:Phigh-Char}, we must have that there is some $g$ such that $\Phi_{g,k}\os=\Phi_{g,k}$. Either $|\Fix(g)|\neq 2$, yielding
	\[\Phi_{g,k}\Xi^1_Y=\Phi_{g,1},\]
or $\Fix(g)=\{i,j\}$ yielding
	\[\Phi_{g,k}\Xi^1_Y=\Phi_{(i\;j),1}.\]
From Lemma~\ref{lem:PhigkOsIsPhigk'}, the right hand side of both equations is fixed by $\os$. Thus,
	\[\Phi_{g,1}=\Phi_{g,1}\os=(\Phi_{g,k}\Xi^1_Y)\os=\Phi_{g,k}\Xi^1_{Y'}\]
or
	\[\Phi_{(i\;j),1}=\Phi_{(i\;j),1}\os=(\Phi_{g,k}\Xi^1_Y)\os=\Phi_{g,k}\Xi^1_{Y'}.\]
That is, $\phi_k$ must be in $Y'$ and so $Y'=\phi_k\au(\sn)=Y$. This completes the proof of the corollary.
\end{proof}
\begin{lemma}\label{lem:PhigkostoPhigk'conj}
	Let $\os$ in $\Ker(\mathfrak{R})$. Then for any $\Phi_{g,k}$ we have that $\Phi_{g,k}\os=\Phi_{g,k'}$ where $k'$ is conjugate to $k$.
\end{lemma}
\begin{proof}
	By Lemma~\ref{lem:PhigkOsIsPhigk'}, we know that $\Phi_{g,k}\os=\Phi_{g,k'}$ for some $\phi_{k'}$ in $A$ and that $\Phi_{g,1}\os=\Phi_{g,1}$. Let $Y=\phi_k\au(\sn)$. Since $\Phi_{g,k}\Delta^1_{Y}=\Phi_{g,1}$, and (by Lemma~\ref{lem:DeltasFixed}) $\Delta^1_{Y}$ is fixed by $\os$, we get that
    \[\Phi_{g,1}=\Phi_{g,1}\os=(\Phi_{g,k}\Delta^1_{Y})\os=\Phi_{g,k}\os\Delta^1_{Y}\os=\Phi_{g,k'}\Delta^1_{Y}.\]
Thus, $\Phi_{g,k'}\Delta^1_{Y}=\Phi_{g,1}$, which implies that $\phi_{k'}$ is in $Y=\phi_k\au(\sn)$, so that $k$ and $k'$ are conjugate.
\end{proof}
\begin{lemma}\label{lem:AllPhigkFixed}
	Let $\os\in\Ker(\mathfrak{R})$. Then $\os$ fixes any element $\Phi_{g,k}$.
\end{lemma}
\begin{proof}
	Consider $\Phi_{g,k}$; we know that if $k=1$ then $\Phi_{g,1}\os=\Phi_{g,1}$. We therefore suppose that $k\neq 1$ and $\Phi_{g,k}\os=\Phi_{g,k'}$ where $k$ and $k'$ are not equal and $k'\neq 1$. Then, by Lemma~\ref{lem:PhigkostoPhigk'conj} we have that $k$ and $k'$ are conjugate. Using Lemma~\ref{lem:SameCentUnderAut} with $M=\en_2(\sn)$, and arguing as in Lemma~\ref{lem:Phigh-Char}, have that $C(g)\cap C(k)=C(g)\cap C(k')$.

  We now make use of Lemma~\ref{lem:conjugateAndCentraliserImpliesEqual} to find $w\in\sn\setminus\an$ of order two such that in $w\in C(k)\setminus C(k')$. In particular, $\phi_w\neq\phi_g$.
	
	Now, let $E_1$ and $E_2$ be non-empty subsets of $E$ that satisfy:
	\begin{enumerate}[(i)]
		\item $E_1$ and $E_2$ are unions of conjugacy classes in $E$;
		\item $E_1\cup E_2=E$; and
		\item $\phi_g\in E_2$.
	\end{enumerate}
	
	Define $\Phi^1_S$ by $E_1S=\{\phi_w\}$, $E_2S=\{\phi_g\}$, $AS=\{\phi_k\}$ and $\phi_1S=\phi_1$. From Lemma~\ref{PhiTopIndex} we have that $\Phi^1_S\os=\Phi^1_T$ for some $T$ where from Lemma~\ref{PhiRank1}, $T$ is not constant and from Lemma~\ref{ResToE}, that $T$ agrees with $S$ on $E$. Then,
	\[\Phi_{g,k}\Phi^1_S=\Phi_{gS,kS}=\Phi_{g,k}.\]
Applying $\os$ yields
	\[\Phi_{g,k'}=\Phi_{g,k}\os=(\Phi_{g,k}\Phi^1_S)\os=\Phi_{g,k}\os\Phi^1_S\os=\Phi_{g,k'}\Phi^1_T=\Phi_{gT,k'T}=\Phi_{gT,kT},\]
Thus, $kT=k'$.

	Now, let $\phi_p$ be in $E_1$. Then,
	\[\Phi_{p,k}\Phi^1_S=\Phi_{pS,kS}=\Phi_{w,k}.\]
Suppose that $\Phi_{p,k}\os=\Phi_{p,l}$ so that $l$ is conjugate to $k$. Applying $\os$ yields
	\[\Phi_{w,k}\os=(\Phi_{p,k}\Phi^1_S)\os=\Phi_{p,k}\os\Phi^1_S\os=\Phi_{p,l}\Phi^1_T=\Phi_{pT,lT}=\Phi_{w,kT}=\Phi_{w,k'}.\]
Thus, by Lemma~\ref{lem:SameCentUnderAut}, we must have that $C(\Phi_{w,k})=C(\Phi_{w,k'})$ and consequently $C(w)\cap C(k)=C(w)\cap C(k')$. But this forces $w\in C(k')$, a contradiction.
	
	Thus, all $\Phi_{g,k}$ are fixed by $\os$.
\end{proof}
\begin{lemma}
	Let $\os\in\Ker(\mathfrak{R})$. Then $\os$ fixes any element in $\Phi$.
\end{lemma}
\begin{proof}
	By Lemma~\ref{PhiTopIndex} it is enough to show that $\os$ fixes all $\Phi^1_S$; we have already done this if $S$ is constant (see Lemma \ref{PhiRank1}), or for maps of the form $\Phi_{g,k}$ (see Lemma \ref{lem:AllPhigkFixed}).
	
	Consider $\Phi^1_S\in \Phi$ where $S$ is not constant and let $\Phi^1_S\os=\Phi^1_T$, so that $T$ is not constant, and $S$ and $T$ agree on $E$ (by Lemma \ref{ResToE}). Let $\phi_k\in A$ and pick any $\phi_g\in E$. Since $\Phi_{g,k}\Phi^1_S=\Phi_{gS,kS}$, we have that
	\[\Phi_{gS,kS}=\Phi_{gS,kS}\os=(\Phi_{g,k}\Phi^1_S)\os=\Phi_{g,k}\os\Phi^1_S\os=\Phi_{g,k}\Phi^1_T=\Phi_{gT,kT}.\]
This forces $\phi_kT=\phi_kS$ and consequently, $S$ and $T$ agree on $A$. Thus $S=T$ as required.
\end{proof}
\begin{theorem}\label{thm:AutEnd2IsomSn}
	The automorphism group of $\en_2(\sn)$ is isomorphic to $\sn$. It consists only of inner automorphisms.
\end{theorem}
\begin{proof}
	We had already seen that $\au(\en_2(\sn))\cong\Ker(\mathfrak{R})\times\sn$. The series of lemmas above show that if $\os$ is in $\Ker(\mathfrak{R})$ then it must be the identity map. Thus, $\Ker(\mathfrak{R})\times\sn\cong\sn$.
\end{proof}
\section{Conclusion and further directions}\label{sec:thoughts}
	We began this work in the knowledge that, for the first two steps of the endomorphism tower of $\sn$ (for $n\geq 7$), the groups of units were again $\sn$. Of course, at the first step ($\en_0(\sn)=\sn$) this is trivial. At the second ($\en_1(\sn)$) it follows easily since the structure of $\en_1(\sn)$ is straightforward to calculate from knowledge of the normal subgroups of $\sn$ and the order 2 elements of $\sn$, and one has to hand the classic result that $\sn\cong \au(\sn)$. The determination that $\au(\en_1(\sn))\cong\sn$ and the description of $\en_2(\sn)$ is much more complex, but again relies on properties of the order 2 elements of $\sn$ and considerations of centralisers. The same is true of the determination that $\au(\en_2(\sn))\cong\sn$. Given these results, the first question we pose is an obvious one. To answer this, one would first have to discover some pattern in the elements of $\en_i(\sn)$ (one that passes from one level to the next, rather than relying on the ground level), that has so far eluded us.
\begin{question}\label{qn:always}
	Is $\au(\en_i(\sn))$ isomorphic to $\sn$ for $i\geq 3$?
\end{question}

We began with a natural semigroup, indeed a group, $\sn$.
\begin{question}\label{qn:otherfinite}
	What can we say about the endomorphism tower of other natural monoids of transformations?
\end{question}

Work has already begun on the endomorphism tower of $\tn$ \cite{Schein:98,Gould:26} at the lowest level. We know the endomorphism tower of a finite semigroup never stabilises in a finite number of steps. However, we can ask the following.
\begin{question}\label{qn:ham2}
	Is there an infinite semigroup $S$ such that the endomorphism tower of $S$ stabilises after a finite number of steps?
\end{question}
\begin{question}\label{qn:ham3}
	For certain monoids $M$, such as $\sn$, can we form infinite endomorphism towers (for this one needs natural morphisms $\alpha:\en_i(M)\rightarrow\en_{i+1}(M)$)? In this case, does the endomorphism tower of a finite, or of an infinite, $M$ stabilise?
\end{question}

This article has focused on semigroup endomorphisms. For a monoid $M$ we let $\en^1(M)$ denote the monoid of {\em monoid} endomorphisms of $M$; the {\em monoid endomorphism tower} is then defined as for the endomorphism tower, but using the monoids $\en_i^1(M)$.
\begin{question}\label{qn:semigroups}
	How do our results translate to monoid endomorphism towers?
\end{question}
\begin{proposition}
	Let $M$ be a finite monoid such that $M$ is isomorphic to $\en^1(M)$. Then $M$ is either the trivial monoid $\{1\}$ or is the trivial monoid with an adjoined zero $\{1,0\}$. These are the only monoids with a terminating monoid endomorphism tower.
\end{proposition}
\begin{proof}
	If $M=\{1\}$ is the trivial monoid then its only monoid endomorphism is the identity map. Thus $M$ is isomorphic to $\en^1(M)$.
	
	Now notice that for any monoid $M$, the constant map to the identity of $M$ is a zero element for $\en^1(M)$. Thus, if $M$ is a non-trivial monoid that is isomorphic to $\en^1(M)$ then $M$ must have a zero element. If $M$ only contains the identity and a zero then it is easy to verify that $M$ is isomorphic to $\en^1(M)$.

	It only remains to show that if $M$ is a finite monoid containing at least 3 elements then $M$ is not isomorphic to $\en^1(M)$. For any idempotent $e$ in such a monoid we can define an idempotent $\lambda_{e}$ in $\en^1(M)$ by
		\[\lambda_{e}:x\to\begin{cases}
			1	&\text{if }x\text{ is in the group of units of }M\\
			e	&\text{otherwise.}
		\end{cases}\]
	Thus, there are at least $|E(M)|$ idempotents in $\en^1(M)$. All the elements of the form $\lambda_e$ are either rank 1 or rank 2 and so are not the identity map on $M$. Thus, $|E(M)|<|E(\en^1(M))|$, which implies that $M$ is not isomorphic to $\en^1(M)$.
\end{proof}
There remains from the group case a question of Hamkins \cite{hamkins:98}.
\begin{question}\label{qn:ham}
	Is there a finite group $G$ such that the automorphism tower of $G$ does not stabilise in a finite number of steps?
\end{question}


\begin{thebibliography}{99}
\bibitem{Adjith:2025} M.V. Ajith, P. J. Cameron, M. Ghosh and A. Lakshmanan S., Endomorphism and automorphism graphs of finite groups. \url{https://arxiv.org/abs/2511.15602}.
\bibitem{Araujo:09} J. Ara\'{u}jo and J Konieczny, General theorems on automorphisms of semigroups and their application, {\em J. Australian Math. Soc.} {\bf 87} (2009), 1–17.
\bibitem{araujo:11} J. Ara\'{u}jo, V. H. Fernandes, M. M. Jesus, V. Maltcev and J. D. Mitchell, Automorphisms of partial endomorphism semigroups, {\em Publ. Math. Debrecen} {\bf 79} (2011), 23--39.
\bibitem{chubb:08} J. Chubb, V.S. Harizanov, A.S. Morozov, S. Pingrey and E. Ufferman, Partial automorphism semigroups, {\em Ann. Pure Appl. Logic} {\bf 156} (2008) 245--258.
\bibitem{Fong:1980} Y. Fong and J. D. P. Meldrum, The endomorphism near-rings of the symmetric groups of degree at least five, {\em J. Australian Math. Soc.} {\bf 30} (1980), 37--49.
\bibitem{Gould:26} V. Gould, A. Grau and M. Johnson, The structure of $\en(\tn)$, {\em Monatsh. Math.} {\bf 209} (2026), 1--37. 
\bibitem{GGJK:2024} V. Gould, A. Grau, M. Johnson and M. Kambites, Translational hulls of semigroups of endomorphisms of an algebra, {\em J. Pure Appl. Algebra} {\bf 230} (2026) 108203.
\bibitem{hamkins:98} J. D. Hamkins, Every group has a terminative transfinite automorphism tower, {\em Proc. American Math. Soc.} {\bf 126} (1998), 3223--3226.
\bibitem{hou:07} H. Hou, Y. Luo, Z. Cheng, The endomorphism monoid of $\overline{P_n}$, {\em European J. Combinatorics} {\bf 29} (2008), 1173--1185.
\bibitem{Just:1999} W. Just, S. Shelah and S. Thomas, The automorphism tower problem revisited, {\em Adv. Math.} {\bf 148} (1999), 243--265.
\bibitem{mazorchuk:02} W. Mazorchuk, Endomorphisms of $\mathfrak{B}_n$, $\mathfrak{PB}_n$, and $\mathfrak{C}_n$, {\em Comm. Algebra} {\bf 30} (2002), 3489--3513.
\bibitem{Thomas:1985} S. Thomas, The Automorphism Tower Problem, {\em Proc. American Math. Soc.} {\bf 95} (1985), 166--168. 
\bibitem{Weilant} H. Wielandt, Eine Verallgemeinerung der invarianten Untergruppen, {\em Math. Z.} {\bf 45} (1939), 209--244.
\bibitem{Schein:98} B. Schein and B. Teclezghi, Endomorphisms of finite full transformation semigroups {\em Proc. American Math. Soc.} {\bf 126} (1998), 2579--2587.
\end{thebibliography}
\end{document}